\documentclass[12pt,reqno]{amsart}

\usepackage{amsmath,amssymb,amsfonts,amsthm}
\usepackage{mathrsfs}
\usepackage{cases}
\usepackage{epic}

\usepackage{graphicx}

\usepackage{upgreek}
\usepackage{bm}
\usepackage{latexsym,todonotes}
\usepackage{pdflscape}
\usepackage[all]{xy}
\usepackage{color}
\usepackage{colordvi}
\usepackage{multicol}
\usepackage[normalem]{ulem}
\usepackage{quiver}

\usepackage{geometry}
\usepackage{fancyhdr}

\newcommand{\arxiv}[1]{\href{http://arxiv.org/abs/#1}{\tt arXiv:\nolinkurl{#1}}}

\numberwithin{equation}{section}
\usepackage[linktocpage=true]{hyperref}
\hypersetup{colorlinks,linkcolor=blue,urlcolor=cyan,citecolor=blue}

\newtheorem{innercustomthm}{{\bf Theorem}}
\newenvironment{customthm}[1]
{\renewcommand\theinnercustomthm{#1}\innercustomthm}
{\endinnercustomthm}
\newcommand{\ie}{{\em i.e.}}

\allowdisplaybreaks

\xyoption{all}

\newcommand{\gim}{\operatorname{gim}\nolimits}

\renewcommand{\mod}{\operatorname{mod}\nolimits}

\newcommand{\rad}{\operatorname{rad}\nolimits}

\newcommand{\Aut}{\operatorname{Aut}\nolimits}

\newcommand{\Mod}{\operatorname{Mod}\nolimits}

\newcommand{\End}{\operatorname{End}\nolimits}

\newcommand{\colim}{\operatorname{colim}\nolimits}
\newcommand{\gldim}{\operatorname{gl.dim}\nolimits}
\newcommand{\cone}{\operatorname{cone}\nolimits}
\newcommand{\rep}{\operatorname{rep}\nolimits}
\newcommand{\Ext}{\operatorname{Ext}\nolimits}

\newcommand{\Hom}{\operatorname{Hom}\nolimits}

\renewcommand{\deg}{\operatorname{deg}\nolimits}

\renewcommand{\dim}{\operatorname{dim}\nolimits}

\newcommand{\diag}{{\operatorname{diag}\nolimits}}

\newcommand{\id}{\operatorname{id}\nolimits}

\def \ov{\overline}

\newcommand{\go}{{K_0}}

\newtheorem{theorem}{Theorem}[section]

\newtheorem{corollary}[theorem]{Corollary}

\newtheorem{example}[theorem]{Example}

\newtheorem{lemma}[theorem]{Lemma}

\newtheorem{proposition}[theorem]{Proposition}
\newtheorem{remark}[theorem]{Remark}

\def \gcd{\mathrm{gcd}}

\def \Z{{\Bbb Z}}

\begin{document}
	\title[GIM and Elliptic Lie algebras via Ringel--Hall Lie algebras]{GIM and Elliptic Lie algebras via Ringel--Hall Lie algebras}
	
	\author[Changjian Fu]{Changjian Fu}
	
	\address{Changjian Fu\\Department of Mathematics, Sichuan University, 610064 Chengdu, P.R.China}
	\email{changjianfu@scu.edu.cn}
	\author[Zhanhong Liang]{Zhanhong Liang}
	\address{Zhanhong Liang\\ Department of Mathematics, Sichuan University, 610064 Chengdu, P.R.China}
	\email{zhanhongliang@stu.scu.edu.cn}

	\author[Ming Lu]{Ming Lu}
	\address{Ming Lu\\Department of Mathematics, Sichuan University, 610064 Chengdu, P.R.China}
	\email{luming@scu.edu.cn}

	
	\dedicatory{Dedicated to Professor Yanan Lin on the Occasion of his 70th Birthday}
	
	\begin{abstract}
		
		For any symmetrizable generalized intersection matrix (GIM) $C$, we construct an acyclic valued quiver $(Q,\mathbf{d})$ endowed with an involution $\theta$. Let $\mathcal{D}$ be the bounded derived category of finite-dimensional representations of $(Q,\mathbf{d})$, and let $\Sigma$ stand for the suspension functor of $\mathcal{D}$. We show that the orbit category $\mathcal{D}/(\theta\circ\Sigma)$ carries a canonical triangulated structure and is $2$-periodic. Applying Peng--Xiao's construction to this orbit category, we prove that the GIM algebra $\operatorname{gim}(C)$ is isomorphic to the integral Ringel--Hall Lie algebra associated with $\mathcal{D}/(\theta\circ\Sigma)$.
		
		As a further application of the above machinery, we investigate elliptic Lie algebras of types $D_4^{(1,1)}$, $E_6^{(1,1)}$, $E_7^{(1,1)}$ and $E_8^{(1,1)}$. For each elliptic Dynkin diagram, we define a finite-dimensional algebra $A$ by taking an appropriate quotient of the acyclic quiver $Q$ attached to the GIM matrix $C$. From the resulting $2$-periodic triangulated categories, we build the corresponding Ringel--Hall Lie algebras, and establish a surjective Lie algebra homomorphism from each elliptic Lie algebra to its integral Ringel--Hall counterpart. This map is conjectured to be injective, and its injectivity on real root spaces is confirmed.

	\end{abstract}

	\maketitle
	
	\pagestyle{fancy}
	\fancyhf{}
	\fancyhead[LE]{\thepage}
	\fancyhead[CE]{\small CHANGJIAN FU, ZHANHONG LIANG AND MING LU}
	\fancyhead[CO]{\small GIM AND ELLIPTIC LIE ALGEBRAS VIA RINGEL--HALL LIE ALGEBRAS}
	\fancyhead[RO]{\thepage}
	\renewcommand{\headrulewidth}{0pt}
	
	\tableofcontents

	\section{Introduction}
	\subsection{Hall algebra realization of Kac--Moody algebras}
	
	Let $A=(a_{ij})\in M_{n\times n}(\Z)$ be a symmetrizable generalized Cartan matrix (GCM for short). The associated Kac-Moody algebra $\mathfrak{g}$ is a complex Lie algebra generated by the Chevalley generators $\{e_i,f_i,h_i\}_{i\in I}$, where $I=\{1,2,\dots,n\}$. This algebra admits a canonical triangular decomposition \[\mathfrak{g}=\mathfrak{g}^+\oplus \mathfrak{h}\oplus\mathfrak{g}^-,\] where $\mathfrak{h}$ is the Cartan subalgebra generated by $\{h_i\}_{i\in I}$, and $\mathfrak{g}^+$ and $\mathfrak{g}^-$ are the positive and negative subalgebras generated by $\{e_i\}_{i\in I}$ and $\{f_i\}_{i\in I}$, respectively.

	The interplay between the representation theory of quivers and Lie theory originates from the pioneering work of Gabriel \cite{Gab72}. 
	For a finite quiver $Q$ without loops, one can associate a symmetric GCM $A=(a_{ij})$ whose entries are defined by
	\begin{align*}
		a_{ij}=\begin{cases}
			2& \text{ if }i=j,
			\\
			-\sharp\{\text{arrows between }i \text{ and }j\}& \text{ if }i\neq j.
		\end{cases}
	\end{align*}
	Gabriel \cite{Gab72} discovered that the classification of quivers of representation-finite type coincides with that of simply-laced semisimple Lie algebras. Moreover, the dimension vector of quiver representations provides a canonical bijection between the isomorphism classes of indecomposable representations of an $ADE$-type quiver $Q$ and the positive roots of the corresponding simple Lie algebra $\mathfrak{g}$.

	This classical result was subsequently generalized by Dlab and Ringel \cite{DR75} to the setting of representation-finite valued quivers. A valued quiver consists of a quiver $Q$ together with a positive integer-valued vertex valuation $\mathbf{d}\colon Q_0\rightarrow\mathbb{Z}_{>0}$. Each valued quiver naturally gives rise to a symmetrizable GCM, and conversely, every symmetrizable GCM can be realized in this manner. Later,  
	Kac \cite{Kac80} further extended this correspondence between indecomposable representations and positive roots of Lie algebras to arbitrary quivers. Since then, exploring the connections between the representation theory of algebras and Lie theory has remained a central and active research topic.

	A landmark breakthrough in this field is the construction of Ringel--Hall algebras introduced by Ringel \cite{R90-Hall-poly-rep-fin}. Let $\mathbb{F}$ be a finite field with $q$ elements, and set $v=\sqrt{q}$. Let $(Q,\mathbf{d})$ be a valued quiver of Dynkin type, and denote by $\rep_\mathbb{F}(Q,\mathbf{d})$ (or simply $\rep(Q)$ by abuse of notation) the abelian category of finite-dimensional representations of $(Q,\mathbf{d})$ over $\mathbb{F}$. Ringel defined the Hall algebra $\mathcal{H}_q(\rep(Q))$ of the category $\rep(Q)$, which by definition is an associative algebra endowed with a basis indexed by the isomorphism classes of quiver representations. The structure constants of this algebra encode the enumerative information of  extension spaces between the representations of $Q$. 
	It was subsequently verified that after specializing the quantum parameter $v$ to $\sqrt{q}$, the twisted Hall algebra $\mathcal{H}_v(\rep(Q))$ is isomorphic to the positive part $\mathbf{U}_v(\mathfrak{g}^+)$ of the quantized enveloping algebra $\mathbf{U}_v(\mathfrak{g})$ of the corresponding simple Lie algebra $\mathfrak{g}$.

	Ringel further introduced the generic Hall algebra for valued quivers of Dynkin type, observing that the structure constants of $\mathcal{H}_v(\rep(Q))$ are polynomial functions of $\sqrt{q}$ when the ground field $\mathbb{F}$ varies, now known as Hall polynomials \cite{Rin90}, also written as functions of $v$.
	By specializing the generic parameter $v\mapsto 1$, the generic Hall algebra admits a natural Lie subalgebra spanned by the isomorphism classes of indecomposable representations, which is isomorphic to the positive part $\mathfrak{g}^+$ of the simple Lie algebra $\mathfrak{g}$. In particular, this isomorphism can be viewed as a lifting of Gabriel's bijection. 
	
	These breakthrough constructions inspired a series of influential generalizations by Green \cite{Gre95}, Riedtmann \cite{Rie94}, Peng--Xiao \cite{PX97-Root-cat-simple-Lie,PX00-Tri-Cat-Kac}, and others. Additionally, Ringel's Hall algebra construction motivated Lusztig \cite{Lus90,Lus91} to establish the geometric realization of $\mathbf{U}_v(\mathfrak{g}^+)$ via perverse sheaves on quiver varieties, which ultimately leads to the construction of canonical bases; parallel foundational work on global crystal bases was developed by Kashiwara \cite{Kas93}. 
	
	
	Despite these profound achievements, Ringel’s original
	construction is restricted to abelian (module) categories and only recovers the positive part $\mathfrak{g}^+$ of a Kac--Moody algebra, failing to capture the full Lie algebra structure including the negative part $\mathfrak{g}^-$ and the Cartan subalgebra $\mathfrak{h}$. 
	To overcome this limitation, Peng and Xiao \cite{PX97-Root-cat-simple-Lie,PX00-Tri-Cat-Kac} developed the theory of Ringel--Hall Lie algebras for $2$-periodic triangulated categories. A prototypical example of such a category is the root category, a notion first introduced by Happel \cite{H87-der-cat-fd-alg}. By definition, a root category is the $2$-periodic orbit category of the bounded derived category of a finite-dimensional hereditary algebra. By applying the construction of Ringel--Hall Lie algebras to root categories of acyclic valued quivers, Peng and Xiao achieved a complete categorical realization of arbitrary symmetrizable Kac--Moody algebras \cite{PX00-Tri-Cat-Kac}.
	
	\subsection{GIM  algebras}
	In the study of intersection forms of exceptional divisors in the resolution of isolated surface singularities, Slodowy \cite{S84-Beyond-Kac-Moody-alg-inside} introduced the notions of generalized intersection matrix (GIM for short) algebras and  intersection matrix algebras, which represent far-reaching generalizations of Kac--Moody algebras. These structures serve as a crucial bridge connecting singularity theory, hyperbolic Kac--Moody theory, and mathematical physics. In a manner analogous to the Kac--Moody theory, given a GIM $C=(c_{ij})\in M_{n\times n}(\Z)$, the associated GIM algebra $\gim(C)$ is defined as a complex Lie algebra generated by the standard generators $\{\tilde{e}_i, \tilde{f}_i, \tilde{h}_i\}_{i \in I}$ subject to a specific set of defining relations.
	
	
	A classical result due to Berman \cite{B89-GIM-alg} states that any GIM algebra can be realized as the fixed-point subalgebra of a suitable Kac--Moody algebra under an involutive automorphism. While this embedding theorem reveals an intrinsic connection between GIM algebras and Kac-Moody algebras, a direct categorical realization of GIM algebras via the Ringel--Hall Lie algebra machinery has remained elusive (cf. \cite{P02,Fu12} for some attempts). This conspicuous gap in the literature serves as the primary motivation for the present work. 
	
	In this paper, we develop a new categorical framework based on  valued quivers with involutions. 
	Given any symmetrizable GIM matrix $C$, we explicitly construct an acyclic valued quiver $Q$ equipped with an involution $\theta$; see Subsection \ref{ss:valued-quiver-(C,D)}. 
	Let $\mathcal{D}$ denote the bounded derived category of finite-dimensional representations of this valued quiver $Q$, with $\Sigma$ its suspension functor. The involution $\theta$ of the quiver $Q$ naturally induces a triangulated automorphism of $\mathcal{D}$, which is also denoted by $\theta$. 
	We consider the orbit category $\mathcal{D}/(\theta\circ \Sigma)$ associated to the composite endofunctor $\theta\circ\Sigma$. Applying Keller’s orbit category theory \cite{Ke05-tri-orb-cat}, we verify that this orbit category is a Hom-finite $2$-periodic triangulated category. By applying the Peng--Xiao's Ringel–Hall Lie algebra construction \cite{PX00-Tri-Cat-Kac} to this new triangulated category, we obtain our first main result, which provides a complete categorical realization of arbitrary symmetrizable GIM  algebras.

	\begin{customthm}{{\bf A}}  [Theorem \ref{thm:GIM-via-RH-Lie}]
		\label{customthm:A}
		Let $C$ be a symmetrizable GIM, and let $\mathscr{L}\mathscr{C}(\mathcal{D}/(\theta\circ \Sigma))$ be the integral Ringel--Hall Lie algebra of the orbit category $\mathcal{D}/(\theta\circ\Sigma)$. Then we have a Lie algebra isomorphism
		\[
		\gim(C) \cong \mathscr{L}\mathscr{C}(\mathcal{D}/(\theta\circ \Sigma))\otimes_{\mathbb{Z}}\mathbb{C},
		\]
		which sends
		\[
		\tilde{e}_i\mapsto \tilde{\bf u}_{S_i},\quad \tilde{f}_i\mapsto -\tilde{\bf u}_{S_{\bar{i}}}, \quad\tilde{h}_i\mapsto \frac{\tilde{\bf h}_{S_i}}{\tilde{d}_{S_i}},\quad 1\leq i\leq n.
		\]
	\end{customthm}
	
	This isomorphism also yields a categorification of Berman's embedding from GIM algebras into Kac--Moody algebras.
	
	
	\subsection{Elliptic Lie algebras}
	Motivated by the geometry of isolated surface singularities, particularly simple elliptic singularities, Saito \cite{Sa85} introduced the notion of elliptic root systems, which were classified via elliptic Dynkin diagrams. Saito and Yoshii \cite{SY00} further introduced an elliptic Lie algebra attached to each simply-laced elliptic Dynkin diagram (cf. \cite{Y04} for the non-simply-laced cases). These Lie algebras turn out to be $2$-toroidal algebras in the sense of Moody, Rao, and Yokonuma \cite{MRY90} and exhibit a close relationship with the extended affine Lie algebras defined in \cite{AABGP97}; cf. also \cite{FP10,FP14}.

	As profound generalizations of finite and affine Kac–Moody algebras, it is natural to ask whether a Ringel–Hall Lie algebra realization of elliptic Lie algebras can be established. In a pioneering work, Lin and Peng \cite{LP05} successfully provided a Ringel–Hall Lie algebra realization for the elliptic Lie algebras of types $D_4^{(1,1)}$, $E_6^{(1,1)}$, $E_7^{(1,1)}$, and $E_8^{(1,1)}$ via the root categories of tubular algebras of types $\mathbb{T}(2,2,2,2)$, $\mathbb{T}(3,3,3)$, $\mathbb{T}(4,4,2)$, and $\mathbb{T}(6,3,2)$, respectively (cf. also \cite{ChL12} for the type $F_4^{(2,2)}$). Despite this breakthrough, a direct categorical realization for elliptic Lie algebras of other types remains a long-standing open problem in the literature.
	
	It is known that elliptic Lie algebras can be realized as quotient algebras of GIM algebras by the ideal generated by the root spaces associated with roots whose squared lengths are greater than $2$; see \cite{SY00}. In the present work, we apply our aforementioned categorical framework developed for GIM algebras to investigate elliptic Lie algebras. We focus on the cornerstone cases of types $D_4^{(1,1)}$, $E_6^{(1,1)}$, $E_7^{(1,1)}$, and $E_8^{(1,1)}$. For each such elliptic Dynkin diagram, we construct a finite-dimensional algebra $A$ via a suitable quotient of the quiver algebra originally used to realize the corresponding GIM algebra. We then establish a $2$-periodic triangulated category $\mathcal{M}$ as the triangulated hull of the orbit category of the bounded derived category $\mathcal{D}^b(\mod A)$. Our second main theorem establishes a categorical correspondence for these elliptic Lie algebras within this novel setting.
	


	\begin{customthm}{{\bf B}}[Theorem \ref{thm:relization-ell-Lie}]
		\label{customthm:B}
		Let $\mathfrak{g}_{\mathrm{ell}}$ be an elliptic Lie algebra of type $D_4^{(1,1)}$, $E_6^{(1,1)}$, $E_7^{(1,1)}$ or $E_8^{(1,1)}$, and let $\mathcal{M}$ be the corresponding $2$-periodic triangulated category. There exists a surjective Lie algebra homomorphism
		\begin{align*}
			\Theta:\mathfrak{g}_{\mathrm{ell}} &\longrightarrow \mathscr{LC}(\mathcal{M})\otimes_{\mathbb{Z}}\mathbb{C},\\
			e_i \longmapsto {\bf u}_{S_i},\quad
			e_{-i} &\longmapsto -{\bf u}_{ S_{\bar i}},\quad
			\alpha_i \longmapsto {\bf h}_{S_i},
		\end{align*}
		for all $i\in I$.
		Moreover, the homomorphism $\Theta$ preserves the $\mathbf{Q}$-grading, and its restriction to each root space $(\mathfrak{g}_{\mathrm{ell}})_\alpha$ is injective for all $\alpha\in R^{\mathrm{re}}\cup\{0\}$.
	\end{customthm}
	This homomorphism preserves root lattice gradation. It is injective on the Cartan subalgebra and all real root spaces, while the injectivity on imaginary root spaces remains an open problem.

	\subsection{Organization}
	The remainder of this paper is organized as follows. Section \ref{s:Preliminaries} collects necessary preliminary definitions and background results, including basic facts on Kac--Moody algebras, valued quivers, root categories and Ringel--Hall Lie algebras. In Section \ref{s:2-periodic-M}, we construct a family of $2$-periodic triangulated categories from algebras equipped with involutions, and discuss their fundamental properties. In Section \ref{s:GIM-RH-Lie}, we associate an acyclic valued quiver with an involution to each symmetrizable GIM matrix, and prove Theorem \ref{customthm:A} which identifies GIM algebras with integral Ringel--Hall Lie algebras of certain orbit categories.  Section \ref{sec:ell-Lie} focuses on the realization of elliptic Lie algebras via Ringel--Hall Lie algebras. By taking suitable quotients of the aforementioned quivers, we build the corresponding $2$-periodic triangulated categories and establish a surjective homomorphism from each elliptic Lie algebra to its associated integral Ringel--Hall Lie algebra (Theorem \ref{customthm:B}).
	
	\subsection*{Convention}
	Let $\mathcal{T}$ be either a triangulated category or an abelian category, we denote by $\go(\mathcal{T})$ its Grothendieck group. For any object $M\in \mathcal{T}$, we denote its image in $\go(\mathcal{T})$ by $\hat{M}$. In the specific case where $\mathcal{T}$ is a $2$-periodic triangulated category, we shall adopt the notation $h_M := \hat{M}$. We denote by
	$\operatorname{ind}\mathcal{T}$ the set of representatives of the isoclasses of all indecomposable objects in $\mathcal{T}$. When $\mathcal{T}=\mod H$, the category of finite-dimensional right $H$-modules for some finite-dimensional algebra $H$, we also write $\operatorname{ind} H:=\operatorname{ind} \mathcal{T}$.
	
	\subsection*{Acknowledgments}
	This work is partially supported by the National Natural Science Foundation of China (No. 12171333, 12471037, 12571040). 
	
	\section{Preliminaries}\label{s:Preliminaries}
	In this section, we recall the construction in \cite{PX00-Tri-Cat-Kac}, in which Kac--Moody algebras are realized via the Ringel--Hall Lie algebras of root categories.
	
	\subsection{Kac--Moody algebras}
	A {\em generalized Cartan
		matrix} (GCM for short) $A =(a_{ij})\in {M}_{n\times n}(\mathbb{Z})$ is an integer matrix satisfying the following conditions:
	\begin{itemize}
		\item[](GCM1) $a_{ii}=2$ for all $i=1,\ldots,n$;
		\item[](GCM2) $a_{ij}\leq 0$ for $i\neq j$;
		\item[](GCM3) $a_{ij}=0$ if and only if $a_{ji}=0$.
	\end{itemize}
	The {\em Kac--Moody algebra} ${\rm gcm}(A)$ associated with $A$ is a complex Lie algebra on $3n$ generators $\{e_i,f_i,h_i \mid 1\leq i\leq n\}$ and  the following defining relations:
	\begin{itemize}
		\item[](1) $[h_i,h_j]=0$, $1\leq i,j\leq n$;
		\item[](2) $[h_i,e_j]=a_{ij}e_j$, $[h_i,f_j]=-a_{ij}f_j$, $1\leq i,j\leq n$;
		\item[](3) $[e_i,f_j]=\delta_{i,j}h_i$, $1\leq i,j\leq n$;
		\item[](4) $({\rm ad} e_i)^{1-a_{ij}}e_j=0$, $({\rm ad} f_i)^{1-a_{ij}}f_j=0$, $1\leq i\neq j\leq n$.
	\end{itemize}
	Here $({\rm ad}x)y:=[x,y]$ for any $x,y$ in a (Lie) algebra.
	
	In this paper, we always assume that the generalized Cartan matrix $A$ is symmetrizable, that is, there is a diagonal matrix $D=\diag(d_1,\dots, d_n)$ with positive integers $d_1,\dots,d_n$ such that $DA$ is symmetric. 
	
	\begin{remark}
		Our convention of Kac--Moody algebras follows that  in \cite{B89-GIM-alg}. In \cite{Kac90}, a different definition of Kac–Moody algebras is adopted. However, when $A$ is symmetrizable, 
		the two definitions coincide.
	\end{remark}

	\subsection{Valued quivers and valued representations}\label{ss:valued-quiver}
	Let $\mathbb{F}$ be a finite field and $\overline{\mathbb{F}}$ its algebraic closure. For each positive integer $r$, denote by $\mathbb{F}_r$ the unique degree $r$ extension of $\mathbb{F}$ contained in $\overline{\mathbb{F}}$. For $\ell\mid r$, we fix a basis of $\mathbb{F}_r$ over $\mathbb{F}_\ell$, which allows us to identify $\mathbb{F}_r$ with an $\mathbb{F}_\ell$-vector space in a canonical way.
	
	Let $Q=(Q_0,Q_1,s,t)$ be a finite quiver, where $Q_0$ and $Q_1$ stand for the sets of vertices and arrows of $Q$, respectively, and $s,t\colon Q_1\rightarrow Q_0$ are the source and target maps.
	A valued quiver is a pair $(Q,\mathbf{d})$ consisting of a finite quiver $Q$ together with a vertex valuation function $\mathbf{d}\colon Q_0\rightarrow\mathbb{Z}_{>0}$. For each vertex $i\in Q_0$, we write $d_i:=\mathbf{d}(i)$.
	A {\it valued representation} of $(Q,\mathbf{d})$ over $\mathbb{F}$ consists of tuples $(V_i,V_\alpha)_{i\in Q_0,\alpha\in Q_1}$, where 
	\begin{itemize}
		\item $V_i$ is an $\mathbb{F}_{d_i}$-vector space for each vertex $i\in Q_0$;
		\item $V_\alpha:V_{s(\alpha)}\rightarrow V_{t(\alpha)}$ is an $\mathbb{F}_{\gcd(d_{s(\alpha)},d_{t(\alpha)})}$-linear map for each arrow $\alpha\in Q_1$, where $\gcd(d_{s(\alpha)},d_{t(\alpha)})$ is the greatest common divisor of $d_{s(\alpha)}$ and $d_{t(\alpha)}$.
	\end{itemize}
	
	Let $W=(W_i,W_\alpha)_{i\in Q_0,\alpha\in Q_1}$ be another valued representation of $(Q,\mathbf{d})$. A {\it homomorphism} from $V$ to $W$
	is a collection $f=(f_i)_{i\in Q_0}$, where $f_i:V_i\rightarrow W_i$ is an $\mathbb{F}_{d_i}$-linear map, such that  the following diagram commutes for each arrow $\alpha\in Q_1$:
	\[
	\xymatrix{V_{s(\alpha)}\ar[d]_{f_{s(\alpha)}}\ar[r]^{V_\alpha}&V_{t(\alpha)}\ar[d]^{f_{t(\alpha)}}\\
		W_{s(\alpha)}\ar[r]^{W_{\alpha}}&W_{t(\alpha)}.}
	\]
	
	A valued representation $V$ of $(Q,\mathbf{d})$ is {\it finite-dimensional} if $\sum_{i\in Q_0}\dim_{\mathbb{F}}V_i<\infty$. Denote by $\operatorname{Rep}_{\mathbb{F}}(Q,\mathbf{d})$ the category of valued representations of $(Q,\mathbf{d})$ over $\mathbb{F}$. Let $\operatorname{rep}_{\mathbb{F}}(Q,\mathbf{d})$ be the full subcategory of $\operatorname{Rep}_{\mathbb{F}}(Q,\mathbf{d})$ consisting of all finite-dimensional valued representations. Denote by $\operatorname{ind}_{\mathbb{F}}(Q,\mathbf{d})$ the set of representatives of the isoclasses of all indecomposable representations in $\rep_{\mathbb{F}}(Q,\mathbf{d})$.
	For the sake of brevity, we shall omit $\mathbb{F}$ and $\mathbf{d}$ from the notation when there is no risk of confusion. Both categories $\operatorname{Rep}(Q)$ and $\operatorname{rep}(Q)$ are hereditary abelian categories.

	From now on we assume that $Q$ is acyclic, \ie, $Q$ contains no nontrivial oriented cycles.  
	It is well known that there is
	a finite-dimensional hereditary $\mathbb{F}$-algebra $H$ such that the category $\mod H$ of finite-dimensional right $H$-modules is equivalent to $\rep(Q)$, see \cite{Ru11} for instance.
	
	For each vertex $i\in Q_0$, let $S_i$ be the simple (valued) representation of $(Q,\mathbf{d})$ associated to $i$, \ie, we assign $\mathbb{F}_{d_i}$ to $i$ and the zero vector space  to every other vertex. For convenience, we denote the Grothendieck group $\go(\rep(Q))$ by $\go(Q)$. Then $\{\hat{S}_i\mid i\in Q_0\}$ forms a $\mathbb{Z}$-basis of $\go(Q)$. The Euler--Ringel bilinear form $\langle-,-\rangle$ on $\go(Q)$ is given by 
	\[
	\langle \hat{M},\hat{N}\rangle =\dim_{\mathbb{F}}\Hom(M,N)-\dim_{\mathbb{F}}\Ext^1(M,N).
	\]
	Denote by $n_{ij}$ the number of arrows from $i$ to $j$. Since we have assumed that $Q$ is acyclic, at most one of $n_{ij}$ and $n_{ji}$ is nonzero.
	We have
	\[
	\langle \hat{S_i},\hat{S_j}\rangle=\begin{cases}
		d_i &\text{if $i=j$};\\
		-\frac{d_id_jn_{ij}}{\gcd(d_i,d_j)}&\text{if  $n_{ij}>0$};\\
		0 &\text{else}.
	\end{cases}
	\]
	Let $(-,-)$ be the symmetric Euler--Ringel form on $\go(Q)$, that is, \[(\hat{M},\hat{N}):=\langle \hat{M},\hat{N}\rangle+\langle \hat{N},\hat{M}\rangle\] for any $M,N\in \rep(Q)$. Then, on the $\mathbb{Z}$-basis $\{\hat{S}_i\mid i\in Q_0\}$, we have
	\[
	( \hat{S_i},\hat{S_j})=\begin{cases}
		2d_i &\text{if $i=j$};\\
		-\frac{d_id_j(n_{ij}+n_{ji})}{\gcd{(d_i,d_j)}} &\text{else}.
	\end{cases}
	\]
	Denote by $Q_0=\{1,\ldots,n\}$. 
	Let $A_Q=(a_{ij})\in M_{n\times n}(\mathbb{Z})$, where
	\begin{equation}\label{eq:gcm-valued-quiver}
		a_{ij}=\frac{2(\hat{S}_i,\hat{S}_j)}{(\hat{S}_i,\hat{S}_i)}=\begin{cases}
			2&\text{if $i=j$};\\
			-\frac{d_j(n_{ij}+n_{ji})}{\gcd(d_i,d_j)}&\text{else}.
		\end{cases}
	\end{equation}
	It follows that $A_Q$ is a GCM with a symmetrizer $D=\operatorname{diag}(d_1,\ldots, d_n)$.
	
	Conversely, for any given symmetrizable GCM $A=(a_{ij})\in M_{n\times n}(\mathbb{Z})$ with a symmetrizer $D=\operatorname{diag}(d_1,\ldots, d_n)$, we may associate an acyclic valued quiver $(Q(A),\mathbf{d})$ as follows:
	\begin{itemize}
		\item The vertex set $Q(A)_0$ is $\{1,\ldots, n\}$;
		\item The valuation $\mathbf{d}:Q(A)_0\rightarrow \mathbb{Z}_{>0}$ is given by $\mathbf{d}(i)=d_i$ for $1\leq i\leq n$;
		\item For $i<j$, we draw $\frac{|a_{ij}|\cdot\gcd(d_i,d_j)}{d_j}$ arrows from $i$ to $j$.
	\end{itemize}
	A direct computation shows that $(\hat{S}_i,\hat{S}_j)=d_ia_{ij}$ for $i\neq j$, which proves $A_{Q(A)}=A$.
	
	\subsection{Root categories}\label{ss:root-cat}
	We begin by establishing the notation and reviewing several preliminary notions. Let $\mathbb{F}$ be a field. Let $\mathcal{T}$ be an $\mathbb{F}$-linear triangulated category with suspension functor $\Sigma$, and let
	$F:\mathcal{T}\rightarrow\mathcal{T}$ be an autoequivalence. The {\em orbit category}
	$\mathcal{T}/F$ has the same objects as $\mathcal{T}$, and its morphism space
	from $X$ to $Y$ is given by
	\[
	\bigoplus_{p\in \mathbb{Z}} \Hom_{\mathcal{T}}(X,F^pY).
	\]
	The composition of morphisms in $\mathcal{T}/F$ is defined in a natural way, and there is
	a canonical projection functor
	$\pi:\mathcal{T}\longrightarrow \mathcal{T}/F$.  We say that the orbit category $\mathcal{T}/F$ admits a canonical triangulated structure if the projection functor $\pi:\mathcal{T}\rightarrow \mathcal{T}/F$ is a triangulated functor.
	
	The triangulated category $\mathcal{T}$ is called {\em $2$-periodic} if
	$\mathcal{T}$ is Krull--Schmidt and $\Sigma^2\cong \id$, where
	$\id:\mathcal{T}\rightarrow\mathcal{T}$ denotes the identity functor. 
	The symmetric Euler bilinear form $(-,-)_{\mathcal{T}}$ over $\go(\mathcal{T})$ for a $2$-periodic triangulated category $\mathcal{T}$ is defined as
	\begin{align}\label{eq:sym-Euler-form}
		\begin{split} (h_{X} ,h_{Y})_{\mathcal{T}}=&\operatorname{dim}_{\mathbb{F}} \operatorname{Hom}_{\mathcal{T}}(X, Y)-\operatorname{dim}_{\mathbb{F}} \operatorname{Hom}_{\mathcal{T}}(X, \Sigma Y) \\ &+\operatorname{dim}_{\mathbb{F}} \operatorname{Hom}_{\mathcal{T}}(Y, X)-\operatorname{dim}_{\mathbb{F}} \operatorname{Hom}_{\mathcal{T}}(Y, \Sigma X). 
		\end{split}
	\end{align}
	
	We now turn our attention to root categories. Let $H$ be a
	finite-dimensional hereditary $\mathbb{F}$-algebra. For instance, one may take
	$H$ to be a hereditary algebra such that the category $\mod H$ is equivalent to $\rep(Q)$ for an acyclic
	valued quiver $Q$; see Subsection~\ref{ss:valued-quiver}. Let $\mathcal{D}^b(H)$ be the bounded derived category of $\mod H$, and its suspension functor is denoted by $\Sigma$. The \emph{root category}
	of $H$ is defined to be the orbit category
	$\mathcal{D}^b(H)/\Sigma^2 .$
	It admits a canonical triangulated structure and, in particular, forms a
	$2$-periodic triangulated category. This was first proved by Peng--Xiao
	\cite{PX97-Root-cat-simple-Lie} using the homotopy category of $2$-periodic
	complexes of projective $H$-modules. Alternatively, this property also follows from a more general result by Keller
	\cite{Ke05-tri-orb-cat} (see Subsection~\ref{ss-Tri-orb-Cat} below).
	
	As in Subsection~\ref{ss:valued-quiver}, let
	$\{S_1,\dots,S_n\}$ be a complete set of representatives of the isomorphism
	classes of simple $H$-modules. Then
	$\{\hat{S}_1,\dots,\hat{S}_n\}$
	forms a $\mathbb{Z}$-basis of the Grothendieck group $\go(H):=\go(\mod H)$.
	It is well known that, via the natural embedding
	$\mod H \hookrightarrow \mathcal{D}^b(H),$
	we may identify $\go(\mathcal{D}^b(H))$ with $\go(H)$.
	


	The canonical projection functor $\pi_H:\mathcal{D}^b(H)\rightarrow \mathcal{D}^b(H)/\Sigma^2$ induces a surjective homomorphism of groups $\hat{\pi}_H:\go(\mathcal{D}^b( H))\rightarrow \go(\mathcal{D}^b(H)/\Sigma^2)$, where $\hat{\pi}_H(\hat{M})=h_{M}$ for $M\in \mathcal{D}^b( H)$. The following is well known (cf. \cite[Proposition 2.11]{Fu12}).
	\begin{lemma}\label{lem:iso-groth-root-cat}
		The homomorphism $\hat{\pi}_H:\go(\mathcal{D}^b( H))\rightarrow \go(\mathcal{D}^b(H)/\Sigma^2)$ is an isomorphism.
	\end{lemma}
	
	As a consequence, $\{h_{S_1},\ldots, h_{S_n}\}$ forms a $\mathbb{Z}$-basis of $\go(\mathcal{D}^b(H)/\Sigma^2)$.
	As noted above, the category $\mathcal{D}^b(H)/\Sigma^2$ admits a symmetric Euler bilinear form $(-,-)_{\mathcal{D}^b(H)/\Sigma^2}$. One readily verifies that $(h_{S_i}, h_{S_i})_{\mathcal{D}^b(H)/\Sigma^2}$ divide $(h_{S_i}, h_{S_j})_{\mathcal{D}^b(H)/\Sigma^2}$ for all $i, j$. Moreover, \[(h_{S_i},h_{S_j})_{\mathcal{D}^b(H)/\Sigma^2}=(\hat{S_i},\hat{S_j}),\] where $(-,-)$ is the symmetric Euler--Ringel bilinear form on $\go( H)$; see Subsection \ref{ss:valued-quiver}.
	Consequently, the matrix $A_H = (a_{ij})_{n \times n}$ defined by
	\[a_{ij} = \frac{2(h_{S_i}, h_{S_j})_{\mathcal{D}^b(H)/\Sigma^2}}{(h_{S_i}, h_{S_i})_{\mathcal{D}^b(H)/\Sigma^2}}\]
	is a GCM with a symmetrizer $D=\operatorname{diag}(d_{S_1},\ldots, d_{S_n})$, where $d_{S_i}=\frac{(\hat{S}_i,\hat{S}_i)}{2}$ for $1\leq i\leq n$.
	
	\subsection{Ringel-Hall Lie algebra}\label{ss-R-H-Lie}
	In this subsection, we recall the definition of Ringel--Hall Lie algebras as constructed by Peng and Xiao \cite{PX00-Tri-Cat-Kac}.
	Let $\mathbb{F}$ be a finite field with $|\mathbb{F}|=q$. Let $\mathcal{T}$ be a $\Hom$-finite $\mathbb{F}$-linear triangulated category with suspension functor $\Sigma$. 
	
	For any objects $X,Y,Z\in\mathcal{T}$, we define
	\[
	\begin{aligned} W(X, Y ; L)= & \left\{(f, g, h) \in \operatorname{Hom}_{\mathcal{T}}(X, L) \times \operatorname{Hom}_{\mathcal{T}}(L, Y) \times \operatorname{Hom}_{\mathcal{T}}(Y, \Sigma X)\mid\right. \\ & X \xrightarrow{f} L \xrightarrow{g} Y \xrightarrow{h} \Sigma X \text { is a triangle}\} .\end{aligned}
	\]
	Denote by $\operatorname{Aut}(X)$ the automorphism group of $X\in\mathcal{T}$. 
	The action of $\operatorname{Aut}(X) \times \operatorname{Aut}(Y)$ on $W(X, Y ; L)$ gives the orbit space
	\[
	V(X, Y ; L)=\left\{(f, g, h)^{\wedge} \mid(f, g, h) \in W(X, Y ; L)\right\},
	\]
	where
	\[
	(f, g, h)^{\wedge}=\left\{\left(fa,  c^{-1}g, (\Sigma a)^{-1}h c\right) \mid(a, c) \in \operatorname{Aut}(X) \times \operatorname{Aut}(Y)\right\} .
	\]
	
	Let $\Hom_{\mathcal{T}}(L,Y)_{\Sigma X}$ be the subset of $\Hom_{\mathcal{T}}(L,Y)$ consisting of morphisms $l:L\rightarrow Y$ whose mapping cone $\operatorname{Cone} (l)$ is isomorphic to $\Sigma X$. Considering the action of the group $\Aut(Y)$ on $\Hom_{\mathcal{T}}(L,Y)_{\Sigma X}$ by $d\cdot l=d\circ l$, the orbit of $l$ is denoted by $l^*$, and the orbit space is denoted by $\Hom_{\mathcal{T}}(L,Y)_{\Sigma X}^*$. Dually, one can also consider the subset $\Hom_{\mathcal{T}}(X,L)_Y$ of $\Hom_{\mathcal{T}}(X,L)$ with the group action of $\Aut (X)$ and the orbit space $\Hom_{\mathcal{T}}(X,L)_Y^*$. According to \cite[Proposition 4.1]{XX08-Hall-tri-cat}, we have the following result. 
	\begin{proposition}\label{pro:Hall number}
		For any $X,Y,L\in\mathcal{T}$, we have 
		$$|V(X,Y;L)|=|\Hom_{\mathcal{T}}(X,L)_Y^*|=|\Hom_{\mathcal{T}}(L,Y)_{\Sigma X}^*|.$$
	\end{proposition}

	We define 
	\begin{align}F_{Y,X}^L:=|V(X,Y;L)|,\text{ and }\gamma_{X,Y}^L:=F_{Y,X}^L-F_{X,Y}^L,
	\end{align}
	where $F_{Y,X}^L$ is called the {\it Hall number} of the triple $(X,L,Y)$.
	
	We further assume that $\mathcal{T}$ is $2$-periodic.

	Let $\mathfrak{h}$ be the subgroup of $\go(\mathcal{T})\otimes_{\mathbb{Z}}\mathbb{Q}$ generated by $\frac{h_M}{d_M}$, $M\in \operatorname{ind} \mathcal{T}$, where 
	\begin{align}
		d_M:=\dim_{\mathbb{F}}\big(\End_{\mathcal{T}}(X)/\rad \End_{\mathcal{T}} (X)\big).
	\end{align} 
	Then the symmetric Euler form defined in \eqref{eq:sym-Euler-form} can be extended to $\mathfrak{h}\times \mathfrak{h}$. Let $\mathfrak{n}$ be the free abelian group with basis $\{ u_X\mid X\in \operatorname{ind} \mathcal{T}\}$. We denote
	\begin{align}
		\mathfrak{g}(\mathcal{T}):=\mathfrak{h} \oplus \mathfrak{n}
	\end{align}
	and  the quotient group
	\begin{align}
		\mathfrak{g}(\mathcal{T})_{(q-1)}:=\mathfrak{g}(\mathcal{T}) /(q-1) \mathfrak{g}(\mathcal{T}).
	\end{align}
	The following remarkable result was proved by Peng--Xiao \cite{PX00-Tri-Cat-Kac}.
	\begin{theorem}[{\cite[Theorem 3.4]{PX00-Tri-Cat-Kac}}]
		The $\mathbb{Z}$-module $\mathfrak{g}(\mathcal{T})_{(q-1)}$ is a Lie algebra over $\mathbb{Z}/(q-1)\mathbb{Z}$ with the Lie operation defined as follows:
		\begin{itemize}
			\item[(1)] For any $X,Y\in\operatorname{ind} \mathcal{T}$,
			\[
			\left[u_{X}, u_{Y}\right]=\sum_{L \in \operatorname{ind} \mathcal{T}}\gamma_{XY}^L u_{L}-\delta_{X, \Sigma Y} \frac{h_{X}}{d_X},
			\]
			where $\delta_{X,\Sigma Y}=1$ for $X\cong \Sigma Y$ and $0$ else;
			\item[(2)] $[\mathfrak{h},\mathfrak{h}]=0$;
			\item[(3)] For any objects $X\in \mathcal{T}$ and $Y\in \operatorname{ind}\mathcal{T}$,
			\[
			[h_X,u_Y]=(h_X,h_Y)_{\mathcal{T}}u_Y,\quad [u_Y,h_X]=-[h_X,u_Y].
			\]
		\end{itemize}
	\end{theorem}
	The Lie algebra $\mathfrak{g}(\mathcal{T})_{(q-1)}$ is called the {\em Ringel-Hall Lie algebra} of $\mathcal{T}$.
	\begin{remark}
		A triangulated category $\mathcal{T}$ is called {\em proper}, if for any nonzero indecomposable object $X\in \mathcal{T}$, we have $\hat{X}\ne 0$ in the Grothendieck group $\go(\mathcal{T})$. In  \cite{PX00-Tri-Cat-Kac}, the Ringel--Hall Lie algebra was defined for proper triangulated category. However, the definition of the Lie operation here is suitable for non-proper triangulated categories (see \cite{X06-der-cat-Lie-alg}), and it coincides with the original definition whenever $\mathcal{T}$ is proper.  
	\end{remark}
	
	\subsection{Kac--Moody algebras via Ringel--Hall Lie algebras}\label{ss:intergral-R-H-Lie-alg}

	For any field extension of $\mathbb{F}\subset \mathbb{K}$, we set $V^\mathbb{K}:=V\otimes_\mathbb{F} \mathbb{K}$ for any $\mathbb{F}$-space $V$. Recall that $H$ is a finite-dimensional hereditary $\mathbb{F}$-algebra. Then $H^\mathbb{K}$ is a $\mathbb{K}$-algebra and for $M\in \mod H$, $M^\mathbb{K}$ has a canonical $H^\mathbb{K}$-module structure. The field $\mathbb{K}$ is called
	{\em conservative} for $X\in \operatorname{ind} H$, if $(\End_H(X)/\rad\End_H(X))^\mathbb{K}$ is a field. Recall that  $\ov{\mathbb{F}}$ is an algebraic closure of $\mathbb{F}$.
	Set
	\[
	\Omega=\{\mathbb{K}\mid\mathbb{F}\subseteq \mathbb{K}\subseteq \ov{\mathbb{F}}\,\, \text{is a finite field extension and conservative for all simple $H$-modules}\}.
	\]
	It is known that $\Omega$ is an infinite set. For any $\mathbb{K}\in\Omega$,  $H^\mathbb{K}$ is also a finite-dimensional hereditary $\mathbb{K}$-algebra with the same symmetric Euler form and minimal symmetrization as the algebra $H$.
	
	Let $\mathfrak{g}((\mathcal{D}^b(H)/\Sigma^2)^\mathbb{K})_{(|\mathbb{K}|-1)}$ be the Ringel-Hall Lie algebra of the root category \[(\mathcal{D}^b(H)/\Sigma^2)^\mathbb{K}:=\mathcal{D}^b(\mod H^\mathbb{K})/\Sigma^2.\] We consider the direct product of the Lie algebras \[\prod_{\mathbb{K}\in \Omega} \mathfrak{g}((\mathcal{D}^b(H)/\Sigma^2)^\mathbb{K})_{(|\mathbb{K}|-1)},\] which has a natural $\mathbb{Z}$-Lie algebraic structure, and denote $\mathscr{LC}(\mathcal{D}^b(H)/\Sigma^2)$ as its Lie subalgebra generated by 
	\begin{align}
		{\bf u}_{S_i}:=(u_{{S_i}^\mathbb{K}})_{\mathbb{K}\in\Omega},\quad
		{\bf u}_{\Sigma S_{i}}:=
		(u_{\Sigma {S_i}^\mathbb{K}})_{\mathbb{K}\in\Omega},\quad
		{\bf h}_{S_i}:=(h_{{S_i}^\mathbb{K}})_{\mathbb{K}\in\Omega}
		,i=1,\dots,n.  
	\end{align}
	The following remarkable result gives a realization of symmetrizable Kac--Moody algebras via Ringel--Hall Lie algebras.
	\begin{theorem}[{\cite[Theorem 4.7]{PX00-Tri-Cat-Kac}}]\label{thm:PX-realization}
		The Kac--Moody algebra ${\rm gcm}(A_H)$ is isomorphic to the integral Ringel--Hall Lie algebra $\mathscr{LC}(\mathcal{D}^b(H)/\Sigma^2)\otimes_{\mathbb{Z}}\mathbb{C}$, where the isomorphism is given by $e_i\mapsto {\bf u}_{S_i}$, $f_i\mapsto -{\bf u}_{\Sigma S_i}$ and $h_i\mapsto \frac{{\bf h}_{S_i}}{d_{S_i}}$ for $1\leq i\leq n$.
	\end{theorem}

	\section{$2$-periodic triangulated categories}\label{s:2-periodic-M} 
	In this section, we introduce a special class of $2$-periodic triangulated categories $\mathcal{M}_\theta$, constructed from involutions $\theta$ on certain finite-dimensional algebras. In the subsequent sections, we shall use their Ringel--Hall Lie
	algebras to study GIM algebras and elliptic Lie algebras. These $2$-periodic triangulated categories arise as triangulated hulls of orbit categories; therefore, we
	begin by recalling Keller's construction of triangulated orbit categories
	\cite{Ke05-tri-orb-cat}.

	
	\subsection{Triangulated orbit categories}\label{ss-Tri-orb-Cat}
	
	Let $A$ be a finite-dimensional $\mathbb{F}$-algebra over some field $\mathbb{F}$, and $\mod A$ be the category of finite-dimensional right $A$-modules. Denote by $\mathcal{D}^b(A)$ the bounded derived category of $\mod A$ with suspension functor $\Sigma$.
	Suppose that $F:\mathcal{D}^b(A)\rightarrow \mathcal{D}^b(A)$ is a {\it standard equivalence}, \ie, $F$ is isomorphic to the derived tensor product
	\[
	? \stackrel{L}{\otimes}_{A} X: \mathcal{D}^{b}(A) \longrightarrow \mathcal{D}^{b}(A)
	\]
	for some complex $X$ of $A$-$A$-bimodules.  In general, the orbit category $\mathcal{D}^b(A)/F$ is not triangulated. To remedy this, Keller introduced a triangulated hull of $\mathcal{D}^b(A)/F$.
	Below, We briefly recall the construction of the triangulated hull following \cite[Section~5]{Ke05-tri-orb-cat}; for the
	terminology concerning differential graded ($=$ dg) categories, we refer to
	\cite{Ke94-dg-cat}.
	
	Let $\mathcal{A}$ be the dg category of bounded complexes of finitely generated
	projective right $A$-modules. In particular, the homotopy category
	$H^0(\mathcal{A})$ is equivalent to $\mathcal{D}^b(A)$. The
	functor
	$F:\mathcal{D}^b(A)\longrightarrow \mathcal{D}^b(A)$
	admits a dg lift
	$\mathcal{F}:\mathcal{A}\longrightarrow \mathcal{A}.$
	The \emph{dg orbit category} $\mathcal{B}$ of $\mathcal{A}$ with respect to
	$\mathcal{F}$ is defined to have the same objects as $\mathcal{A}$, and its
	morphism complexes are given by
	\[
	\mathcal{B}(X,Y)
	\cong
	\colim_{p}\bigoplus_{n\geq 0}
	\mathcal{A}\bigl(\mathcal{F}^nX,\mathcal{F}^pY\bigr).
	\]
	This construction yields
	$H^0(\mathcal{B})\cong H^0(\mathcal{A})/F
	\cong \mathcal{D}^b(A)/F.$ Let $\mathcal{D}\mathcal{A}$ and $\mathcal{D}\mathcal{B}$ denote the derived
	categories of $\mathcal{A}$ and $\mathcal{B}$, respectively. Let
	$\mathcal{M}:=\operatorname{per}(\mathcal{B})$ denote the perfect derived category of $\mathcal{B}$, \ie, the thick subcategory of $\mathcal{DB}$ generated by the representable functors $\Hom_{\mathcal{B}}(-,X)$ for all $X\in \mathcal{B}$. The Yoneda embedding
	$H^0(\mathcal{B})\rightarrow\mathcal{D}\mathcal{B}$
	then induces a canonical embedding
	$\mathcal{D}^b(A)/F \rightarrow \mathcal{M}$. The triangulated category $\mathcal{M}$ is called the {\em triangulated hull} of the orbit category
	$\mathcal{D}^{b}(A)/F$. 
	
	
	Let $\pi:\mathcal{A}\rightarrow \mathcal{B}$ be the canonical dg functor. It yields a $\mathcal{B}\otimes \mathcal{A}^{op}$-module
	\[
	(B, A) \longrightarrow \mathcal{B}(B, \pi A),
	\]
	which induce the standard functors
	\begin{eqnarray}\label{eq:functor-pi-rho}
		\begin{tikzcd}
			{\mathcal{DA}} & {\mathcal{DB}.}
			\arrow["{\pi_*}", shift left, from=1-1, to=1-2]
			\arrow["{\pi_\rho}", shift left, from=1-2, to=1-1]
		\end{tikzcd}
	\end{eqnarray}
	Note that there is a canonical equivalence of triangulated categories $\mathcal{DA}\stackrel{\sim}{\longrightarrow}\mathcal{D}(\operatorname{Mod} A)$, where $\mathcal{D}(\operatorname{Mod} A)$ denotes the derived category of
	right $A$-modules. By identifying $\mathcal{DA}$ with $\mathcal{D}(\operatorname{Mod} A)$, for any $X\in \mathcal{D}^b(\mod A)$, we have
	\begin{align}\label{eq:pi-rho-pi-star-X}
		\pi_\rho\pi_\ast X\cong \bigoplus_{p\in \mathbb{Z}}F^pX.  
	\end{align}
	
	In general, the embedding $\mathcal{D}^b(A)/F \rightarrow \mathcal{M}$ is not dense (cf. \cite{Ke05-tri-orb-cat}).
	However, under certain conditions, this embedding does become dense. The following remarkable result established by Keller \cite{Ke05-tri-orb-cat} provides precise criteria for this density.

	\begin{theorem}[{\cite[Theorem 9.9]{Ke05-tri-orb-cat}}] \label{Thm: tri-orb-cat}
		Let $A$ be a finite-dimensional hereditary $\mathbb{F}$-algebra, and assume that the derived category $\mathcal{D}^b(A)$ along with the standard equivalence $F$ satisfies the following conditions:
		\begin{itemize}
			\item For each indecomposable $U$ of $\mod A$, only finitely many objects $F^iU,i\in \mathbb{Z}$, lie in $\mod A$;
			\item There is an integer $N\ge 0$ such that the $F$-orbit of each indecomposable of $\mathcal{D}^b(A)$ contains an object $\Sigma^p U$ for some $0\le p\le N$ and some indecomposable $U$ of $\mod A$.
		\end{itemize}
		Then the orbit category $\mathcal{D}^b(A)/F$ admits a canonical triangulated structure.
	\end{theorem}

	\subsection{2-periodic triangulated category $\mathcal{M}_\theta$}\label{ss:2-periodic-M}
	In this subsection, we shall always assume that $A$ is a finite-dimensional
	$\mathbb{F}$-algebra of finite global dimension. Let $\theta$ be an involution
	of the algebra $A$. Then $\theta$ induces an exact additive functor on the
	category of finite-dimensional modules,
	$\theta:\operatorname{mod} A \rightarrow \operatorname{mod} A,$
	where, by abuse of notation, we denote both the algebra involution and the
	induced functor by the same symbol $\theta$.
	
	Clearly, one has $\theta^{2}\cong \operatorname{id}_{\operatorname{mod} A}$. Moreover, the functor
	$\theta$ induces a triangle equivalence on the bounded derived category
	$\mathcal{D}^{b}(A)$. Continuing this abuse of notation, we denote the induced
	triangulated functor again by
	$\theta:\mathcal{D}^{b}(A)\rightarrow \mathcal{D}^{b}(A).$
	It follows that
	$\theta^{2}\cong \operatorname{id}_{\mathcal{D}^{b}(A)},$
	and that $\theta$ commutes with the suspension functor $\Sigma$ of
	$\mathcal{D}^{b}(A)$.
	
	Let
	$G:=\theta\circ \Sigma:\mathcal{D}^{b}(A)\rightarrow \mathcal{D}^{b}(A)$
	be the composition of $\theta$ with the suspension functor $\Sigma$. Since both
	$\theta$ and $\Sigma$ are autoequivalences, $G$ is also an autoequivalence of
	$\mathcal{D}^{b}(A)$. Furthermore, $G$ is a standard equivalence. Indeed,
	one can define the $A$-$A$-bimodule $A^{\theta}$. As a vector space, one has
	$A^{\theta}=A$, while its bimodule structure is given by
	$a\cdot x\cdot b=\theta(a)xb$, where $a,b\in A,x\in A^{\theta}.$
	There is an isomorphism of triangulated functors
	\[
	? \stackrel{L}{\otimes}_{A} \Sigma A^{\theta}\cong G
	:\mathcal{D}^{b}(A)\longrightarrow \mathcal{D}^{b}(A).  
	\]
	As explained in Section~\ref{ss-Tri-orb-Cat}, one can define the
	triangulated hull $\mathcal{M}_{\theta}$ of the orbit category
	$\mathcal{D}^{b}(A)/G$. We now discuss some of its properties.
	
	As discussed in Section~\ref{ss-Tri-orb-Cat}, for any object $X\in \mathcal{D}^b(A)$, one has $\pi_\rho X\in \mathcal{D}(\Mod A)$. Moreover, equation~\eqref{eq:pi-rho-pi-star-X}, together with the definition of the functor $G$ and the finite global dimension of $A$, implies that 
	$\pi_\rho X\in \mathcal{D}(\mod A)$. By induction, the same conclusion holds for every object 
	$X\in \mathcal M_\theta$.
	
	In the category of complexes $\mathcal{C}(\mod A)$, an object is called
	a $\theta$-complex if it is a complex
	$X^{\bullet}=(X^i,d^i)$ satisfying
	\[
	X^{i+1}=\theta X^i,
	\quad
	d^{i+1}=\theta d^i, \quad
	\forall  i\in \mathbb{Z}.
	\]
	The morphism between two $\theta$-complexes is a chain map $f=(f^i)_{i\in\mathbb{Z}}$ satisfying
	\[
	f^{i+1}=\theta f^i,\quad \forall i\in \mathbb{Z}.
	\]
	We define $\mathcal{C}_{\theta}(A)$ to be the
	subcategory of $\mathcal{C}(\mod A)$ whose objects are $\theta$-complexes
	and whose morphisms are morphisms of $\theta$-complexes.
	
	A morphism $f\in \Hom_{\mathcal{C}_{\theta}(A)}(X^\bullet,Y^\bullet)$ is
	called null-homotopic if there exists a sequence of morphisms
	$s^i:X^i\rightarrow Y^{i-1}$ such that
	\begin{align*}
		&s^{i+1}=\theta s^i,
		\quad f^i= d_Y^{i-1}s^i+s^{i+1} d_X^i , \quad \forall i\in\mathbb{Z}.
	\end{align*}
	A morphism is called a quasi-isomorphism if it induces an isomorphism on
	homology groups. Similarly to the case of $m$-periodic complexes, one can define
	the relative homotopy category $\mathcal{K}_{\theta}(A)$ and the relative
	derived category $\mathcal{D}_{\theta}(A)$. Moreover, both
	$\mathcal{K}_{\theta}(A)$ and $\mathcal{D}_{\theta}(A)$ are triangulated
	categories, with suspension functor induced by the shift functor of complexes;
	cf.~\cite{Hap88,PX97-Root-cat-simple-Lie,Sta18}.
	
	Let $\mathcal{C}^b(A)$ be the category of bounded complex of $\operatorname{mod} A$.
	We define an exact functor
	$\widetilde{\pi}_{\theta}:\mathcal{C}^b(A)\rightarrow \mathcal{C}_{\theta}(A)$
	as follows:
	\begin{align*}
		\widetilde{\pi}_{\theta}(X^{\bullet})
		&=
		\left(
		\bigoplus_{i\in\mathbb{Z}} X^{2i+j}
		\oplus
		\theta\bigl(X^{2i+1+j}\bigr),
		\begin{bmatrix}
			d^{2i+j} & 0 \\
			0 & \theta\bigl(d^{2i+1+j}\bigr)
		\end{bmatrix}_{i\in\mathbb{Z}}
		\right)_{j\in\mathbb{Z}}, \\
		\widetilde{\pi}_{\theta}(f)
		&=
		\left(
		\begin{bmatrix}
			f^{2i+j} & 0 \\
			0 & \theta\bigl(f^{2i+1+j}\bigr)
		\end{bmatrix}_{i\in\mathbb{Z}}
		\right)_{j\in\mathbb{Z}},
	\end{align*}
	where $X^\bullet\in \mathcal{C}^b(A)$ and
	$f\in \Hom_{\mathcal{C}^b(A)}(X^\bullet,Y^\bullet)$. The functor
	$\widetilde{\pi}_{\theta}$ induces a triangulated functor
	$\mathcal{K}^b(A)\rightarrow \mathcal{K}_{\theta}(A),$
	and hence a triangulated functor $\overline{\pi}_{\theta}:\mathcal{D}^b(A)\rightarrow \mathcal{D}_{\theta}(A)$ between derived categories. We have the following result; see also \cite[Theorem 2.12]{LR24}.
	
	\begin{lemma}\label{lem:pi-theta-fully-fathful}
		There exists a triangulated functor
		$\pi_{\theta}:\mathcal{M}_{\theta}\rightarrow \mathcal{D}_{\theta}(A)$
		such that the following diagram commutes:
		\[
		\begin{tikzcd}
			\mathcal{D}^{b}(A)
			\arrow[r, "\overline{\pi}_{\theta}"]
			\arrow[d, "\mathrm{can.}"']
			&
			\mathcal{D}_{\theta}(A)
			\\
			\mathcal{M}_{\theta}.
			\arrow[ur, "\pi_{\theta}"']
			&
		\end{tikzcd}
		\]
		Moreover, the functor $\pi_{\theta}$ is fully faithful.
	\end{lemma}
	
	\begin{proof}
		By definition, there is a natural isomorphism of functors
		$\overline{\pi}_{\theta}\circ G \simeq \overline{\pi}_{\theta}$. Hence \cite[\S 9.4]{Ke05-tri-orb-cat} implies that
		there exists a triangulated functor
		$\pi_{\theta}:\mathcal{M}_{\theta}\rightarrow \mathcal{D}_{\theta}(A)$
		such that the above diagram commutes.
		
		By definition, for any $X,Y\in \mathcal{D}^b(A)$, one has
		\[
		\bigoplus_{p\in \mathbb{Z}}\Hom_{\mathcal{D}^b(A)}(X,G^pY)\cong \Hom_{\mathcal{D}_\theta(A)}(\bar{\pi}_\theta(X),\bar{\pi}_\theta(Y)).
		\] 
		It follows directly 
		from the commutative diagram above that the restriction of $\pi_{\theta}$ to
		representable functors in $\mathcal{M}_\theta$ is fully faithful. Consequently, $\pi_{\theta}$ is
		fully faithful on $\mathcal{M}_{\theta}$. This completes the proof.
	\end{proof}
	For any object $X\in\mod A$, by the definition of the functor $\pi_\theta$, the image $\pi_\theta X$ can be represented as the following $\theta$-complex:
	\begin{align}\label{eq:pi-theta-object}
		\cdots\rightarrow \theta X \xrightarrow{d_M^{-1}=0}X\xrightarrow{d_M^0=0} \theta X\rightarrow \cdots.
	\end{align}
	In general, let
	$f\in \Hom_{\mathcal{M}_{\theta}}(X,Y)$
	be a morphism in $\mathcal{M}_{\theta}$. If $\pi_{\theta}X$ is chosen to be
	a $\theta$-complex of projective modules, or alternatively if
	$\pi_{\theta}Y$ is chosen to be a $\theta$-complex of injective modules,
	then $\pi_{\theta}f:\pi_\theta X\rightarrow \pi_\theta Y$ can be represented by the following chain map:
	\begin{align}\label{eq:pi-theta-X-2-periodic-complex}
		\begin{split}
			\xymatrix{
				\cdots &
				{X^{-1}=\theta X^0}\ar[d]_{f^{-1}=\theta f^0}
				\ar[rrr]^{d_X^{-1}=\theta d_X^0}
				&&&
				{X^0}\ar[d]_{f^0}\ar[rr]^{d_X^0}
				&&
				{X^1=\theta X^0}\ar[d]^{f^1=\theta f^0}
				& \cdots \\
				\cdots &
				{Y^{-1}=\theta Y^0}
				\ar[rrr]^{d_Y^{-1}=\theta d_Y^0}
				&&&
				Y^0\ar[rr]^{d_Y^0}
				&&
				{Y^1=\theta Y^0}
				& \cdots.
			}
		\end{split}
	\end{align}

	The functor $\theta\colon \mathcal{D}^b(A)\rightarrow \mathcal{D}^b(A)$ can be naturally extended to the category $\mathcal{M}_\theta$.
	
	\begin{lemma}\label{lem:theta-in-M}
		The functor $\theta:\mathcal{D}^b(A)\rightarrow \mathcal{D}^b(A)$ induces an involution of $\mathcal{M}_\theta$, again denoted by $\theta$. Furthermore, we have a natural isomorphism of functors
		$\theta \cong \Sigma:\mathcal{M}_\theta \rightarrow \mathcal{M}_\theta .$
	\end{lemma}
	\begin{proof}
		Since the functor $\theta:\mathcal{D}^b(A)\rightarrow \mathcal{D}^b(A)$ commutes with $G$, it induces an involution
		$\theta:\mathcal{D}^b(A)/G\rightarrow \mathcal{D}^b(A)/G.$
		Moreover, this involution can be extended to $\mathcal{M}_\theta$. More precisely, the functor
		$\theta:\mathcal{D}^b(A)\rightarrow \mathcal{D}^b(A)$
		admits a dg lift
		$\Theta:\mathcal{A}\rightarrow \mathcal{A},$
		which in turn naturally induces a dg autoequivalence
		$\widetilde{\Theta}:\mathcal{B}\rightarrow \mathcal{B}.$
		Consequently, one obtains an involution
		$\theta:\mathcal{D}\mathcal{B}\rightarrow \mathcal{D}\mathcal{B}$
		on the derived category, and this further restricts to a functor
		$\theta:\mathcal{M}_\theta\rightarrow \mathcal{M}_\theta.$
		It is straightforward to verify that the restriction of the functor $\theta$ on $\mathcal{M}_\theta$ to $\mathcal{D}^b(A)/G$ coincides with the original involution $\theta$ on $\mathcal{D}^b(A)/G$. Therefore, we shall repeatedly use the same notation $\theta$ for these functors without causing ambiguity.
		
		On the other hand, the functor
		$\theta:\mathcal{M}_\theta\rightarrow\mathcal{M}_\theta$
		coincides with the standard equivalence induced by the $\mathcal{B}\otimes \mathcal{B}^{op}$-bimodule
		$\mathcal{B}(X,\theta Y)=\Sigma\mathcal{B}(X,Y).$
		Hence there is a functor isomorphism
		$\theta\cong \Sigma:\mathcal{M}_\theta\rightarrow\mathcal{M}_\theta$ by \cite[Lemma 6.1]{Ke94-dg-cat}. 
	\end{proof}
	
	To construct the Ringel--Hall Lie algebra of $\mathcal{M}_\theta$, we need to
	establish that $\mathcal{M}_\theta$ is a $\Hom$-finite
	$2$-periodic triangulated category. This is the content of the following
	proposition.
	
	\begin{proposition}\label{pro:M_theta-2-periodic}
		Let $A$ be a finite-dimensional $\mathbb{F}$-algebra of finite global
		dimension equipped with an involution $\theta$. Then the triangulated hull
		$\mathcal{M}_{\theta}$ defined above is a $\Hom$-finite $2$-periodic
		triangulated category. 
	\end{proposition}
	
	\begin{proof}
		See \cite[Proposition~2.2]{Fu12}. Although the corresponding result is proved
		there for the root category, the same argument applies in the present setting.
	\end{proof}
	
	For any object $X\in \mathcal{D}^b(A)$, we shall still denote by $X$ its image under the canonical embedding
	$\mathcal{D}^b(A)/G \rightarrow \mathcal{M}_\theta$.
	We have the following description of indecomposable objects in $\mathcal{M}_\theta$.
	
	\begin{lemma}
		Let $X\in \mathcal{D}^b(A)$ be an indecomposable object. Then its image
		$X\in \mathcal{D}^b(A)/G \subseteq \mathcal{M}_\theta$
		is still indecomposable. In particular, if the algebra $A$ is hereditary, then $\mathcal{D}^b(A)/G \simeq \mathcal{M}_\theta$, and the indecomposable objects of $\mathcal{M}_\theta$ are, up to isomorphism, precisely the images of the indecomposable objects of $\mathcal{D}^b(A)$.
	\end{lemma}
	
	\begin{proof}
		To prove the first statement, it suffices to show that for every indecomposable object
		$X \in \mathcal{D}^b(A)$, the endomorphism algebra of $X$ in the orbit category
		$\mathcal{D}^b(A)/G$ is local; see \cite[Proposition~1.2]{BMRRT}. By the definition of the orbit category, we have
		\[
		\End_{\mathcal{D}^b(A)/G}(X)
		=
		\bigoplus_{i\in \mathbb{Z}}
		\Hom_{\mathcal{D}^b(A)}(X,G^iX).
		\]
		Moreover, its radical is given by
		\[
		\rad_{\mathcal{D}^b(A)/G}(X,X)
		=
		\rad_{\mathcal{D}^b(A)}(X,X)
		\oplus
		\bigoplus_{0\neq i\in \mathbb{Z}}
		\Hom_{\mathcal{D}^b(A)}(X,G^iX).
		\]
		Hence
		\[
		\End_{\mathcal{D}^b(A)/G}(X)/
		\rad_{\mathcal{D}^b(A)/G}(X,X)
		\cong
		\End_{\mathcal{D}^b(A)}(X)/
		\rad_{\mathcal{D}^b(A)}(X,X)
		\]
		is a division algebra, which implies the first statement.
		
		Now assume that $A$ is hereditary. By Theorem~\ref{Thm: tri-orb-cat}, the canonical
		embedding
		$\mathcal{D}^b(A)/G \rightarrow \mathcal{M}_\theta$
		is dense. Therefore, for any object $X\in \mathcal{M}_\theta$, there exists an object
		$X'\in \mathcal{D}^b(A)/G$ such that $X\cong X'$ in $\mathcal{M}_\theta$.
		Thus every indecomposable object of $\mathcal{M}_\theta$ is isomorphic to the image of
		an object of $\mathcal{D}^b(A)/G$.
		
		By the first part, the image of every indecomposable object of $\mathcal{D}^b(A)$ is
		indecomposable in $\mathcal{D}^b(A)/G$, and hence in $\mathcal{M}_\theta$. Conversely,
		if an object of $\mathcal{D}^b(A)$ is decomposable, then so is its image in the orbit
		category. Therefore the indecomposable objects of $\mathcal{M}_\theta$ are, up to
		isomorphism, precisely the images of the indecomposable objects of
		$\mathcal{D}^b(A)$.
	\end{proof}

	We denote the Grothendieck group of $\mod A$ and $\mathcal{D}^b(A)$ simply by $\go(A)$ and $\go(\mathcal{D})$ respectively.
	And now we assume that the functor $\theta:\mod A\rightarrow \mod A$ has no fixed simple modules. Since $\theta$ is an involution and it preserves simple modules, the
	simple $A$-modules split into two-element $\theta$-orbits. Thus, we may choose
	one representative $S_i$ from each such orbit, say $S_1,\dots,S_n$, so that
	$\{S_1,\dots,S_n\}\cup \{\theta S_1,\dots,\theta S_n\}$
	is a complete set of representatives of the isomorphism classes of simple
	$A$-modules. The following proposition describes the Grothendieck group $\go(\mathcal{M}_\theta)$ in this case.
	
	\begin{lemma}\label{lem:basis-go-M}
		Assume that $\theta S\ncong S$ for every simple $A$-module $S$. Let
		$S_1,\dots,S_n$ be chosen as above. Then the Grothendieck group
		$\go(\mathcal{M}_\theta)$ is a free abelian group with a basis
		$h_{S_1},\ldots,h_{S_n}.$
	\end{lemma}
	
	\begin{proof}
		The standard functor
		$\pi_*:\mathcal{D}A\rightarrow \mathcal{D}B$
		restricts to a functor
		$\pi_*:\mathcal{D}^b(A)\rightarrow \mathcal{M}_\theta .$
		It induces a surjective group homomorphism
		\begin{align*}
			\hat{\pi}:\go(\mathcal{D})&\longrightarrow \go(\mathcal{M}_\theta),\\
			\hat X&\longmapsto h_{X}.
		\end{align*}
		Since
		$\hat{S}_1,\ldots,\hat{S}_n,
		\hat{\theta S_1},\ldots,\hat{\theta S_n}$
		form a $\mathbb{Z}$-basis of
		$\go(\mathcal{D})=\go(A)$, it follows that
		$h_{S_1},\ldots,h_{S_n}$ and $
		h_{\theta S_1},\ldots,h_{\theta S_n}$
		generate $\go(\mathcal{M}_\theta)$.
		
		On the other hand, by Lemma~\ref{lem:theta-in-M}, we have
		$\Sigma S_i\cong \theta S_i$
		in $\mathcal{M}_\theta$ for each $1\leq i\leq n$. Hence, in the
		Grothendieck group $\go(\mathcal{M}_\theta)$, one has $h_{\theta S_i}=h_{\Sigma S_i}=-h_{S_i}.$
		Therefore, $\go(\mathcal{M}_\theta)$ is generated by
		$h_{S_1},\ldots,h_{S_n}.$
		It remains to prove that these generators are $\mathbb{Z}$-linearly independent.
		
		For any $X\in \mathcal{M}_\theta$, consider
		$\pi_\rho X\in \mathcal{D}(\mod A).$ We denote its zeroth and first cohomology groups by
		$H^0(\pi_\rho X)\quad \text{and}\quad H^1(\pi_\rho X)$ respectively. Then we claim that the assignment
		\begin{align*}
			\phi:\go(\mathcal{M}_\theta)&\longrightarrow \go(\mathcal{D}),\\
			h_X&\longmapsto \widehat{H^0(\pi_\rho X)}-\widehat{H^1(\pi_\rho X)}
		\end{align*}
		is a well-defined group homomorphism.
		
		If this is the case, we have
		\[
		\phi(h_{S_i})=H^0(\pi_\rho S_i)-H^1(\pi_\rho S_i)=H^0(\bigoplus_{p\in \mathbb{Z}} G^p S_i)-H^1(\bigoplus_{p\in \mathbb{Z}} G^p S_i)=
		\hat{S}_i-\hat{\theta S_i}
		\]
		for $1\leq i\leq n$, and the desired result follows.
		
		To show that $\phi$ is well-defined, it suffices to show that $\phi(h_X+h_Z-h_Y)=0$ for any triangle 
		\begin{equation}\label{eq:triangle}
			X\longrightarrow Y\longrightarrow Z\longrightarrow \Sigma X
		\end{equation}
		in $\mathcal{M}_\theta$.
		By applying $\pi_\rho$ to the triangle \eqref{eq:triangle}, we obtain  a triangle in $\mathcal{D}(\mod A)$
		\begin{equation}\label{eq:triangle-DA}
			\pi_\rho X\longrightarrow \pi_\rho Y\longrightarrow \pi_\rho Z\longrightarrow  
			\Sigma (\pi_\rho X).
		\end{equation}
		This induces a long exact sequence in cohomology. In particular, we have
		\begin{align*}
			\operatorname{coker}\bigl( H^{1}(\pi_\rho Y)\rightarrow H^{1}(\pi_\rho Z) \bigr)
			&\cong
			\ker \bigl( H^{2}(\pi_\rho X)\rightarrow H^{2}(\pi_\rho Y) \bigr)\\
			&\cong
			\ker \bigl( H^{0}(\Sigma^2\pi_\rho X)\rightarrow H^{0}(\Sigma^2\pi_\rho Y) \bigr).
		\end{align*}
		Since $\pi_\rho$ is a triangulated functor and $\mathcal{M}_\theta$ is a
		$2$-periodic triangulated category, there are functorial isomorphisms
		$\Sigma^2\circ \pi_\rho \simeq \pi_\rho\circ \Sigma^2 \simeq \pi_\rho.$
		Hence
		\[
		\operatorname{coker}\bigl( H^{1}(\pi_\rho Y)\rightarrow H^{1}(\pi_\rho Z) \bigr)
		\cong
		\ker \bigl( H^{0}(\pi_\rho X)\rightarrow H^{0}(\pi_\rho Y) \bigr).
		\]
		Thus, we obtain the following long exact sequence in $\mod A$:
		\[
		\begin{tikzcd}[
			column sep=0.8cm, 
			row sep=0.8cm,
			arrows={-Stealth, thick}
			]
			\ker f \arrow[r, tail] & 
			H^{0} (\pi_\rho X) \arrow[r,"f"] & 
			H^{0} (\pi_\rho Y) \arrow[r] & 
			H^{0} (\pi_\rho Z) \arrow[lld, out=0, in=-180, looseness=1.5] \\
			& 
			H^{1} (\pi_\rho X) \arrow[r] & 
			H^{1} (\pi_\rho Y) \arrow[r] & 
			H^{1} (\pi_\rho Z) \arrow[r, two heads] & 
			\ker f.
		\end{tikzcd}
		\]
		It follows that 
		\begin{align*}
			\widehat{H^0(\pi_\rho X)}-\widehat{H^1(\pi_\rho X)}+\widehat{H^0(\pi_\rho Z)}-\widehat{H^1(\pi_\rho Z)} 
			-\widehat{H^0(\pi_\rho Y)}+\widehat{H^1(\pi_\rho Y)}=0.
		\end{align*}
		Equivalently,
		$\phi(h_X+h_Z-h_Y)=0.$
		Thus, $\phi$ is well-defined, and the proof is complete.
	\end{proof}

	\section{GIM algebras via Ringel--Hall Lie algebras}\label{s:GIM-RH-Lie}
	
	In this section, given any symmetrizable GIM
	$C=(c_{ij})\in M_{n\times n}(\mathbb{Z})$ with symmetrizer
	$D=\operatorname{diag}(d_1,\ldots,d_n)$, we construct a valued quiver
	$(Q,\mathbf d)$ equipped with an involution $\theta$ associated with $C$.
	We then introduce a corresponding $2$-periodic triangulated category and
	realize the associated symmetrizable GIM algebra via the Ringel--Hall Lie
	algebra of this category.
	
	\subsection{GIM algebras}\label{ss:GIM algebra} In this subsection, we recall the definition of the generalized intersection matrix algebra and the classical result \cite{B89-GIM-alg} that a GIM algebra can be embedded into a Kac--Moody algebra.
	
	A matrix ${C}=(c_{ij})_{n\times n}\in M_{n\times n}(\mathbb{Z})$ is a {\em generalized intersection matrix}  if it satisfies the following:
	\begin{itemize}
		\item[](GIM1) $c_{ii}=2$ for all $i=1,\ldots,n$;
		\item[](GIM2) $c_{ij}< 0$ if and only if $c_{ji}<0$;
		\item[](GIM3) $c_{ij}>0$ if and only if $c_{ji}>0$.
	\end{itemize}
	A GIM $C\in M_{n\times n}(\mathbb{Z})$ is {\em symmetrizable} if there is a diagonal matrix $D=\operatorname{diag}(d_1,\ldots,d_n)$ with positive integers $d_1,\ldots,d_n$ such that $DC$ is symmetric. In this case, $D$ is called a {\em symmetrizer} of $C$.
	It is clear that a GCM is a GIM, but the converse is not true.
	
	For a given GIM $C=(c_{ij})\in M_{n\times n}(\mathbb{Z})$, the {\em GIM algebra} ${\rm gim}(C)$ associated with $C$ is the complex Lie algebra generated by $\{\widetilde{e}_i,\widetilde{f}_i,\widetilde{h}_i\mid 1\le i\le n\}$ subject to the following defining relations:
	\begin{align}
		&[\widetilde{h}_{i}, \widetilde{h}_{j}]=0,\quad 1 \leq i, j \leq n;\label{rel:gim-1}\\
		&[\widetilde{h}_{i}, \widetilde{e}_{j}]=c_{ij} \widetilde{e}_{j},\quad [\widetilde{h}_{i}, \widetilde{f}_{j}]=-c_{ij} \widetilde{f}_{j},\quad 1 \leq i, j \leq n;\label{rel:gim-2}\\
		& [\widetilde{e}_i,\widetilde{f}_i]=\widetilde{h}_i,\quad 1\le i\le n;\label{rel:gim-3}\\
		& [\widetilde{e}_{i}, \widetilde{f}_{j}]=[\widetilde{f}_i,\widetilde{e}_j]=0,\quad 
		(\operatorname{ad} \widetilde{e}_{i})^{-c_{ij}+1} \widetilde{e}_{j}=0= (\operatorname{ad} \widetilde{f}_{i})^{-c_{ij}+1} \widetilde{f}_{j},\quad i \neq j, c_{i j} \leq 0;\label{rel:gim-4}\\
		& [\widetilde{e}_{i},\widetilde{e}_{j}]=[\widetilde{f}_{i}, \widetilde{f}_{j}]=0,\quad 
		(\operatorname{ad} \widetilde{e}_{i})^{c_{ij}+1} \widetilde{f}_{j}=0= (\operatorname{ad} \widetilde{f}_{i})^{c_{ij}+1} \widetilde{e}_{j},\quad\quad\,  i \neq j,
		c_{i j}>0.\label{rel:gim-5}
	\end{align}
	According to the definition, if $C$ is a GCM, then ${\rm gcm}(C)={\rm gim}(C)$.
	
	For any GIM ${C}=(c_{ij})\in M_{n\times n}(\mathbb{Z})$, we associate a $2n\times 2n$ generalized Cartan matrix $A(C)=(a_{ij})$ with indices set $\{1, \ldots, n\} \cup \{\bar{1}, \ldots, \bar{n}\}$ defined as follows:
	\begin{align}
		a_{ii}&=a_{\bar{i},\bar{i}}=2, \quad \text{for $1\leq i\leq n$}; \label{construction-GIM-GCM1}\\
		a_{ij}&=a_{\bar i,\bar j}=\begin{cases}
			c_{ij} & \text{ if } c_{ij}\le 0,\\
			0 & \text{ if } c_{ij}>0,
		\end{cases}
		\qquad \text{ for $i\neq j$};\label{construction-GIM-GCM2}\\
		a_{i\bar j}&=a_{\bar i j}=\begin{cases}
			-c_{ij} & \text{ if } c_{ij}\ge 0, \\
			0 &\text{ if }c_{ij}<0,
		\end{cases} \qquad\text{ for $i\neq j$}.\label{construction-GIM-GCM3}
	\end{align}
	The following is obvious.
	\begin{lemma}
		Let $C\in M_{n\times n}(\mathbb{Z})$ be a symmetrizable GIM with a symmetrizer $D$. Then $A(C)$ is a symmetrizable GCM with the symmetrizer $\begin{bmatrix}D&0\\ 0&D\end{bmatrix}$.
	\end{lemma}
	Let ${\rm gcm}(A(C))$ be the Kac-Moody algebra associated with $A(C)$ with the set of generators \[\{e_i,f_i,h_i\mid i=1, \ldots, n ,\bar{1}, \ldots, \bar{n}\}.\] 
	Let $\sigma$ be the involution of ${\rm gcm}(A(C))$ such that $\sigma(e_i)=f_{\bar{i}}$, $\sigma(f_i)=e_{\bar{i}}$ and $\sigma(h_{i})=-h_{\bar{i}}$ for all $i\in \{1,\ldots, n,\bar{1},\ldots, \bar{n}\}$. Here we adopt the convention that $\bar{\bar{i}}=i$.

	\begin{theorem}[\cite{B89-GIM-alg}]\label{thm:GIM-Kac-Moody}
		Let $C\in M_{n\times n}(\mathbb{Z})$ be a GIM and $A(C)\in M_{2n\times 2n}(\mathbb{Z})$ the associated GCM. The GIM algebra ${\rm gim}(C)$ is isomorphic to the fixed point subalgebra of the involutory automorphism $\sigma$ of ${\rm gcm}(A(C))$. Furthermore, the embedding homomorphism $\iota$ is given by
		\begin{align*}
			\iota:\, &{\rm gim}(C) \longrightarrow {\rm gcm}(A(C)),
			\\
			&\widetilde{e}_i \longmapsto e_i+f_{\bar i},\\
			&\widetilde{f}_i \longmapsto f_i+e_{\bar i},\\
			&\widetilde{h}_i \longmapsto h_i-h_{\bar i},  
		\end{align*}
		for $i=1,\dots, n$.
	\end{theorem}
	\begin{remark}\label{rem:twisted-embedding}
		We can slightly modify the embedding homomorphism $\iota$ in Theorem \ref{thm:GIM-Kac-Moody}. 
		Let $\tau_{1}:{\rm gim}(C)\rightarrow {\rm gim}(C)$ be the automorphism given by \[\tilde{e}_i\mapsto\sqrt{-1}\tilde{e}_i,\quad \tilde{f}_i\mapsto -\sqrt{-1}\tilde{f}_i, \quad \tilde{h}_i\mapsto \tilde{h}_i, \quad\forall i=1,\ldots, n.\] Let $\tau_2:{\rm gcm}(A(C))\rightarrow {\rm gcm}(A(C))$ be the automorphism given by 
		\[
		e_i\mapsto \sqrt{-1}e_i,\quad f_i\mapsto -\sqrt{-1}f_i,\quad h_i\mapsto h_i,\quad i=1,\ldots,n,\bar{1},\ldots, \bar{n}.
		\]
		A direct computation shows that  the embedding $\iota':=\tau_2\circ \iota\circ \tau_1^{-1}:{\rm gim}(C)\to {\rm gcm(A(C))}$ is given by 
		\[
		\widetilde{e}_i \longmapsto e_i-f_{\bar i}, \quad \widetilde{f}_i \longmapsto f_i-e_{\bar i},\quad \widetilde{h}_i \longmapsto h_i-h_{\bar i}.
		\]
		
	\end{remark}
	
	\subsection{Valued quiver associated with symmetrizable GIM $(C,D)$}\label{ss:valued-quiver-(C,D)} Throughout the remainder of this section, we fix a symmetrizable GIM
	$C=(c_{ij})\in M_{n\times n}(\mathbb{Z})$ with symmetrizer
	$D=\operatorname{diag}(d_1,\ldots,d_n)$, and abbreviate this pair by $(C,D)$.
	
	For any $i\neq j$, let $n_{ij}=\frac{|c_{ij}|\cdot\gcd(d_i,d_j)}{d_j}$. We define the valued quiver $Q:=Q(C,D)$ associated with $(C,D)$ as follows:
	\begin{itemize}
		\item The set of vertices is $Q_0=\{1,\ldots, n\}\cup\{\bar{1},\ldots,\bar{n}\}$;
		\item The set of arrows $Q_1$ is constructed based on the sign of $c_{ij}$ for $i<j$:
		\begin{itemize}
			\item if  $c_{ij}<0$, there are $n_{ij}$ arrows from $i$ to $j$, denoted by $\alpha_{ij}^{(1)},\ldots, \alpha_{ij}^{(n_{ij})}$,
			and  $n_{ij}$ arrows from $\bar{i}$ to $\bar{j}$, denoted by $\alpha_{\bar{i},\bar{j}}^{(1)},\ldots, \alpha_{\bar{i},\bar{j}}^{(n_{ij})}$;
			\item if  $c_{ij}>0$, there are $n_{ij}$ arrows from $i$ to $\bar{j}$, denoted by $\alpha_{i\bar{j}}^{(1)},\ldots,\alpha_{i\bar{j}}^{(n_{ij})}$,
			and  $n_{ij}$ arrows from $\bar{i}$ to $j$, denoted by $\alpha_{\bar{i}{j}}^{(1)},\ldots,\alpha_{\bar{i}{j}}^{(n_{ij})}$.
		\end{itemize}
		\item The valuation $\mathbf{d}:Q_0\rightarrow \mathbb{Z}$ is given by $\mathbf{d}(i)=d_i$ and $\mathbf{d}(\bar{i})=d_i$ for $1\leq i\leq n$;
	\end{itemize}
	It is evident from the construction that $(Q,\mathbf{d})$ is acyclic. Furthermore, there exists a natural involution $\theta$ on $(Q, \mathbf{d})$ defined by:
	\[\theta(i) = \bar{i}, \text{ and } \theta(\bar{i}) = i, \; \forall 1\leq i\leq n.\]
	This involution acts on the arrows by swapping the ``barred" and ``unbarred" labels (e.g., $\theta(\alpha_{ij}^{(k)}) = \alpha_{\bar{i},\bar{j}}^{(k)}$ and $\theta(\alpha_{i\bar{j}}^{(k)}) = \alpha_{\bar{i}j}^{(k)}$), effectively preserving the valuation $\mathbf{d}$.
	
	Let $S_1,\ldots, S_n,S_{\bar{1}},\ldots,S_{\bar{n}}$ be the simple representations of $(Q,\mathbf{d})$ over the finite field $\mathbb{F}$.
	Recall that the GCM $A(C)$ associated with $(C,D)$ is defined  as in \eqref{construction-GIM-GCM1}--\eqref{construction-GIM-GCM3}. The following is a consequence of definitions.
	\begin{lemma}\label{lem:Ringel-Euler-form-(Q,d)}
		The matrix of the symmetric Euler--Ringel form $(-,-)$ on the Grothendieck group $\go(Q)$, with respect to  the basis $\hat{S}_1,\ldots,\hat{S}_{\bar{n}}$, coincides with the matrix $\operatorname{diag}(D,D)A(C)$. Consequently, we have $A_Q=A(C)$, where $A_Q$ is the GCM of $(Q,\mathbf{d})$ defined in \eqref{eq:gcm-valued-quiver}.
	\end{lemma}
	
	\begin{example}\label{ep:GIM-quiver}
		We present here an example of a symmetric GIM matrix of order $6$ together with its associated valued quiver. Let
		\[
		C=
		\begin{pmatrix}
			2 & -1 & -1 & -1 & -1 & 2 \\
			-1 & 2 & 0 & 0 & 0 & -1 \\
			-1 & 0 & 2 & 0 & 0 & -1 \\
			-1 & 0 & 0 & 2 & 0 & -1 \\
			-1 & 0 & 0 & 0 & 2 & -1 \\
			2 & -1 & -1 & -1 & -1 & 2 
		\end{pmatrix}.
		\]
		In this case, $D$ is the identity matrix. Hence the valuation is given by
		$\mathbf{d}(i)=1$, for $1\leq i\leq 6.$
		Thus, the valued quiver $(Q,\mathbf d)$ may be regarded as an ordinary quiver, as shown in Figure~\ref{fig:quiver-GIM}.
		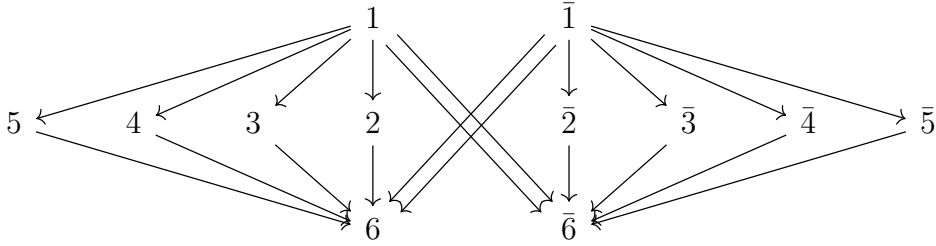
\begin{figure}[htbp]
			\centering
			\begin{tikzcd}
				&&& 1 && {\bar 1} &&& \\
				5 & 4 & 3 & 2 && {\bar 2} & {\bar 3} & {\bar 4} & {\bar 5} \\
				&&& 6 && {\bar 6}
				\arrow[from=1-4, to=2-1]
				\arrow[from=1-4, to=2-2]
				\arrow[from=1-4, to=2-3]
				\arrow[from=1-4, to=2-4]
				\arrow[shift right, from=1-4, to=3-6]
				\arrow[shift left, from=1-4, to=3-6]
				\arrow[from=1-6, to=2-6]
				\arrow[from=1-6, to=2-7]
				\arrow[from=1-6, to=2-8]
				\arrow[from=1-6, to=2-9]
				\arrow[shift left, from=1-6, to=3-4]
				\arrow[shift right, from=1-6, to=3-4]
				\arrow[from=2-1, to=3-4]
				\arrow[from=2-2, to=3-4]
				\arrow[from=2-3, to=3-4]
				\arrow[from=2-4, to=3-4]
				\arrow[from=2-6, to=3-6]
				\arrow[from=2-7, to=3-6]
				\arrow[from=2-8, to=3-6]
				\arrow[from=2-9, to=3-6]
			\end{tikzcd}
			\caption{The valued quiver associated with GIM $C$.}
			\label{fig:quiver-GIM}
		\end{figure}
		Later, we will compare it with the elliptic Cartan matrix of type $D_4^{(1,1)}$ and its corresponding quiver; see Example~\ref{ep:D_4-quiver}.
	\end{example}

	\subsection{2-periodic triangulated categories associated to $(C,D)$}\label{ss:2-periodic-(C,D)}
	Recall that $\rep(Q)=\rep_{\mathbb{F}}(Q,\mathbf{d})$ is a hereditary abelian category. Denote by $\mathcal{D}:=\mathcal{D}^b(\rep(Q))$ the bounded derived category of $\rep(Q)$ with suspension functor $\Sigma$.
	
	The involution $\theta$ on $(Q, \mathbf{d})$ induces an exact additive functor on the category of representations, $\theta: \rep(Q) \rightarrow \rep(Q)$ (by abuse of notation, we denote both the map and the induced functor by $\theta$). For a vauled represetation $V=(V_i,V_\alpha)_{i\in Q_0,\alpha\in Q_1}$, the representation $\theta(V):=(W_i,W_\alpha)_{i\in Q_0,\alpha\in Q_1}$ is defined as follows:
	\begin{itemize}
		\item $W_i=V_{\bar{i}}$ for $i\in Q_0$;
		\item $W_\alpha=V_{\theta(\alpha)}:V_{s(\bar{\alpha})}\rightarrow V_{t(\bar{\alpha})}$ for $\alpha\in Q_1$.
	\end{itemize}
	Clearly, $\theta^2\cong \id$. Moreover,  
	$\theta(S_i)\cong S_{\bar{i}}$ for $i=1,\dots,n$.
	
	As previously noted, there exists a finite-dimensional hereditary algebra $H$
	such that
	$\mod H\simeq \rep(Q).$
	The involution $\theta$ on $(Q,\mathbf d)$ induces an involution of $H$,
	and hence an exact functor on $\mod H$, which we still denote by
	$\theta:\mod H\rightarrow \mod H.$
	Moreover, the equivalence
	$\mod H\simeq \rep(Q)$
	is compatible with the functor $\theta$. Therefore, in what follows, we shall
	identify the categories $\rep(Q)$ and $\mod H$, together with the
	corresponding functor $\theta$ on them.
	
	Therefore, by the construction in Section~\ref{s:2-periodic-M}, we obtain a
	$2$-periodic triangulated category $\mathcal{M}_\theta$. Since
	$\rep(Q)$ is hereditary in the present setting, it follows that
	$\mathcal{D}/G \simeq \mathcal{M}_\theta.$
	Consequently, the orbit category $\mathcal{D}/G$ admits a canonical
	triangulated structure. In what follows, we shall write $\mathcal{D}/G$
	instead of $\mathcal{M}_\theta$.

	


	According to Lemma \ref{lem:basis-go-M}, $h_{S_1},\ldots, h_{S_n}$ is a $\mathbb{Z}$-basis of $\go(\mathcal{D}/G)$. Let $(-,-)_{\mathcal{D}/G}$ be the symmetric Euler form of $\mathcal{D}/G$. Then we have the following.
	\begin{lemma}\label{lem:Euler-form-D/G}
		The matrix of $(-,-)_{\mathcal{D}/G}$ with respect to the basis $h_{S_1},\ldots, h_{S_n}$ of $\go(\mathcal{D}/G)$ is $DC$.
	\end{lemma}
	\begin{proof}
		Since $\rep(Q)$ is hereditary, for any $1\leq i,j\leq n$, we have
		\begin{eqnarray*}
			\Hom_{\mathcal{D}/G}(S_i,S_j)&=&\bigoplus_{p\in \mathbb{Z}}\Hom_{\mathcal{D}}(S_i,G^pS_j)\\
			&=&\Hom_{\mathcal{D}}(S_i,S_j)\oplus \Hom_{\mathcal{D}}(S_i,GS_j)\\
			&=&\Hom_{\mathcal{D}}(S_i,S_j)\oplus \Hom_{\mathcal{D}}(S_i,\Sigma S_{\bar{j}}).
		\end{eqnarray*}
		Similarly,  $\Hom_{\mathcal{D}/G}(S_i,\Sigma S_j)=\Hom_{\mathcal{D}}(S_i,S_{\bar{j}})\oplus \Hom_{\mathcal{D}}(S_i,\Sigma S_j)$.
		Now assume that $1\leq i\leq j\leq n$.  By definition, 
		\begin{eqnarray*}
			(h_{S_i},h_{S_j})_{\mathcal{D}/G}
			&=&\dim \Hom_{\mathcal{D}/G}(S_i,S_j)-\dim \Hom_{\mathcal{D}/G}(S_i,\Sigma S_j)\\
			&&+\dim \Hom_{\mathcal{D}/G}(S_j,S_i)-\dim \Hom_{\mathcal{D}/G}(S_j,\Sigma S_i)\\
			&=& \dim \Hom_{\mathcal{D}}(S_i,S_j)+\dim \Hom_{\mathcal{D}}(S_i,\Sigma S_{\bar{j}})\\
			&&+\dim \Hom_{\mathcal{D}}(S_j,S_i)
			+\dim \Hom_{\mathcal{D}}(S_j,\Sigma S_{\bar{i}})\\
			&&-\dim \Hom_{\mathcal{D}}(S_i,S_{\bar{j}})-\dim \Hom_{\mathcal{D}}(S_i,\Sigma S_j)\\
			&&-\dim \Hom_{\mathcal{D}}(S_j,S_{\bar{i}})-\dim \Hom_{\mathcal{D}}(S_j,\Sigma S_i).
		\end{eqnarray*}
		If $i=j$, then $(h_{S_i},h_{S_i})_{\mathcal{D}/G}=2\dim \Hom_{\mathcal{D}}(S_i,S_i)=2d_i$.  
		If $i<j$, then 
		\begin{eqnarray*}
			&&(h_{S_i},h_{S_j})_{\mathcal{D}/G}\\
			&=&\dim \Hom_{\mathcal{D}}(S_i,\Sigma S_{\bar{j}})+\dim \Hom_{\mathcal{D}}(S_j,\Sigma S_{\bar{i}})-\dim \Hom_{\mathcal{D}}(S_i,\Sigma S_{{j}})-\dim \Hom_{\mathcal{D}}(S_j,\Sigma S_{{i}})\\
			&=&\dim \Hom_{\mathcal{D}}(S_i,\Sigma S_{\bar{j}})-\dim \Hom_{\mathcal{D}}(S_i,\Sigma S_{{j}})\\
			&=&\begin{cases}
				-\dim\Hom_{\mathcal{D}}(S_i,\Sigma S_j)=d_ic_{ij} &\text{if $c_{ij}\leq 0$};\\
				\dim \Hom_{\mathcal{D}}(S_i,\Sigma S_{\bar{j}})=d_ic_{ij}&\text{if $c_{ij}> 0$}.
			\end{cases}
		\end{eqnarray*}
		The proof is completed.
	\end{proof}

	For the valued quiver $(Q,\mathbf d)$, we may also define the corresponding
	root category $\mathcal{D}/\Sigma^2$ as in Subsection~\ref{ss:root-cat}; it is
	again a $2$-periodic triangulated category.
	
	Both $\mathcal{D}/\Sigma^2$ and $\mathcal{D}/G$ have the same objects as $\mathcal{D}$. Up to isomorphism, the indecomposable objects of $\mathcal{D}/G$ coincide with the indecomposable objects of $\rep(Q)$. In contrast, the indecomposable objects of $\mathcal{D}/\Sigma^2$ consist, up to isomorphism, of the indecomposable objects of $\rep(Q)$ together with their one-step suspensions. 
	
	Let $|\mathbb{F}|=q$. Applying the construction in Subsection \ref{ss-R-H-Lie}, we obtain two Lie algebras $\mathfrak{g}(\mathcal{D}/G)_{(q-1)}$ and $\mathfrak{g}(\mathcal{D}/\Sigma^2)_{(q-1)}$ over $\mathbb{Z}/(q-1)\mathbb{Z}$. To avoid confusion between the notation in the two Ringel--Hall Lie algebras
	$\mathfrak{g}(\mathcal{D}/G)_{(q-1)}$ and
	$\mathfrak{g}(\mathcal{D}/\Sigma^2)_{(q-1)}$, we adopt the following convention:
	\begin{itemize}
		\item For $X,Y,L\in \mathcal{D}$, we  denote  by $g_{Y,X}^L$  the Hall number in $\mathcal{D}/G$, and by $h_{Y,X}^L$ the Hall number in $\mathcal{D}/\Sigma^2$. Set $\tilde{\gamma}_{X,Y}^L:=g_{Y,X}^L-g_{X,Y}^L$ and $\gamma_{X,Y}^L=h_{Y,X}^L-h_{X,Y}^L$.
		\item For $X\in\mathcal{D}$, we denote by $\tilde{h}_X$ the image of $X$ in $\go(\mathcal{D}/G)$, and by $h_X$ the image of $X$ in $\go(\mathcal{D}/\Sigma^2)$.
		\item For an indecomposable object $M\in \mathcal{D}$, $\tilde{d}_M:=\dim_{\mathbb{F}} \big(\End_{\mathcal{D}/G}(M)/\rad \End_{\mathcal{D}/G}(M)\big)$ and $d_M:=\dim_{\mathbb{F}} \big(\End_{\mathcal{D}/\Sigma^2}(M)/\rad \End_{\mathcal{D}/\Sigma^2}(M)\big)$.
		\item Let $\tilde{\mathfrak{h}}$ be the subgroup of $\go(\mathcal{D}/G)\otimes_\mathbb{Z} \mathbb{Q}$ generated by $\frac{\tilde{h}_M}{\tilde{d}_M}$ with $M\in \operatorname{ind} \mathcal{D}/G$, and $\mathfrak{h}$ the subgroup of $\go(\mathcal{D}/\Sigma^2)\otimes_{\mathbb{Z}}\mathbb{Q}$ generated by $\frac{h_M}{d_M}$ with $M\in \operatorname{ind} \mathcal{D}/\Sigma^2$.
		
		\item Let $\tilde{\mathfrak{n}}$ be the free abelian group with basis $\{\tilde{u}_X\mid X\in \operatorname{ind} \mathcal{D}/G\}$ and $\mathfrak{n}$ the free abelian group with basis $\{u_X\mid X\in \operatorname{ind} \mathcal{D}/\Sigma^2\}$.
		\item $\mathfrak{g}(\mathcal{D}/G):=\tilde{\mathfrak{h}}\oplus \tilde{\mathfrak{n}}$ and $\mathfrak{g}(\mathcal{D}/\Sigma^2):=\mathfrak{h}\oplus \mathfrak{n}$.
	\end{itemize}

	\subsection{The Ringel--Hall Lie algebra $\mathfrak{g}(\mathcal{D}/G)_{(q-1)}$}\label{ss:sur-inj of R-H-Lie}
	
	The following result is a direct consequence of the definition of $\mathfrak{g}(\mathcal{D}/G)_{(q-1)}$ and Lemma~\ref{lem:theta-in-M}.
	
	\begin{lemma}\label{lem:involution-Lie-algebra}
		The map $X\mapsto \theta(X)$ induces an involution  $\theta:\mathfrak{g}(\mathcal{D}/G)_{(q-1)}\rightarrow \mathfrak{g}(\mathcal{D}/G)_{(q-1)}$. In particular, $\theta(\tilde{h}_X)=\tilde{h}_{\theta(X)}$ and $\theta(\tilde{u}_X)=\tilde{u}_{\theta(X)}$.
	\end{lemma}
	
	The following is also well known, which will be used in the calculation of $\mathfrak{g}(\mathcal{D}/G)_{(q-1)}$. 
	\begin{lemma}\label{lem:s.e.s-induce-tri}
		Let $M\xrightarrow{f} L\xrightarrow{g} N\rightarrow \Sigma M$ be a triangle in $\mathcal{D}$ such that $M,N\in\rep(Q)$. Then $0\rightarrow M\xrightarrow{f} L\xrightarrow{g} N \rightarrow 0$ is a short exact sequence in $\rep(Q)$.
	\end{lemma}

	
	Now we prove that $\mathfrak{g}(\mathcal{D}/G)_{(q-1)}$ satisfies the Serre relations of $\operatorname{gim}(C)$.
	\begin{proposition}\label{pro:Serre-relation}
		The following relations hold in the Lie algebra $\mathfrak{g}(\mathcal{D}/G)_{(q-1)}$:
		\begin{itemize}
			\item[(1)] For $1\leq i\neq j\leq n$ and $c_{ij}\leq 0$, 
			\begin{align}
				&[\tilde{u}_{{S_i}}, \tilde{u}_{{S_{\bar{j}}}}]=[\tilde{u}_{{S_{\bar i}}}, \tilde{u}_{{S_j}}]=0,\\
				&(\operatorname{ad} \tilde{u}_{{S_i}})^{-c_{ij}+1} \tilde{u}_{{S_j}}=0= (\operatorname{ad} \tilde{u}_{{S_{\bar i}}})^{-c_{ij}+1} \tilde{u}_{{S_{\bar j}}};\label{rel:serre-rel-<0}
			\end{align}
			
			\item[(2)] For $1\leq i\neq j\leq n$ and $c_{ij}> 0$,
			\begin{align}
				&[\tilde{u}_{{S_i}},\tilde{u}_{{S_j}}]=[\tilde{u}_{{S_{\bar i}}}, \tilde{u}_{{S_{\bar j}}}]=0,\label{rel:serre-rel->00}\\
				&(\operatorname{ad} \tilde{u}_{{S_i}})^{c_{ij}+1} \tilde{u}_{{S_{\bar j}}}=0= (\operatorname{ad} \tilde{u}_{{S_{\bar i}}})^{c_{ij}+1} \tilde{u}_{{S_j}}.\label{rel:serre-rel->01}
			\end{align}
			
		\end{itemize}
		
	\end{proposition}
	\begin{proof}
		Let us assume that $c_{ij}\leq 0$. It follows that
		\[
		\Hom_{\mathcal{D}/G}(S_i,\Sigma S_{\bar{j}})=\Hom_{\mathcal{D}}(S_i,\Sigma S_{\bar{j}})\oplus \Hom_{\mathcal{D}}(S_i,S_j)=0,
		\]
		and 
		\[
		\Hom_{\mathcal{D}/G}(S_{\bar{j}},\Sigma S_{i})=\Hom_{\mathcal{D}}(S_{\bar{j}},\Sigma S_{i})\oplus \Hom_{\mathcal{D}}(S_{\bar{j}},S_{\bar{i}})=0.
		\]
		Consequently, $\tilde{\gamma}_{S_i,S_{\bar{j}}}^L=0$ for any indecomposable object $L\in \mathcal{D}/G$. On the other hand, we conclude that $S_i\not\cong \Sigma S_{\bar{j}}$ by Lemma \ref{lem:basis-go-M}. Hence $[\tilde{u}_{S_i},\tilde{u}_{S_{\bar{j}}}]=0$. By applying the involution $\theta$ of $\mathfrak{g}(\mathcal{D}/G)_{(q-1)}$ to $[\tilde{u}_{S_i},\tilde{u}_{S_{\bar{j}}}]=0$, we obtain $[\tilde{u}_{S_{\bar{i}}},\tilde{u}_{S_j}]=0$. 
		
		Now we turn to the relation \eqref{rel:serre-rel-<0}. Without loss of generality, we may assume that $i<j$; the case $j<i$ can be proved by a similar argument. Note that $i<j$ implies that \[\Hom_{\mathcal{D}/G}(S_j,\Sigma S_i)=\Hom_{\mathcal{D}}(S_j,\Sigma S_i)\oplus \Hom_{\mathcal{D}}(S_j,S_{\bar{i}})=0.\]
		Consequently, $g_{S_j,S_i}^L=0$ for any indecomposable object $L\in \mathcal{D}/G$. We recursively define collections $\mathcal{X}_t(t\geq 0)$ of objects of $\mathcal{D}$ as follows:
		\begin{itemize}
			\item $\mathcal{X}_0=\{S_j\}$;
			\item Suppose that $\mathcal{X}_{t-1}$ is defined, then 
			\[
			\mathcal{X}_t=\{X\in \mathcal{D}\mid \text{$\exists$ a triangle $X_{t-1}\rightarrow X\rightarrow S_i\rightarrow \Sigma X_{t-1}$ in $\mathcal{D}/G$ with $X_{t-1}\in \mathcal{X}_{t-1}$}\}.
			\]
		\end{itemize}
		We claim that
		\begin{enumerate}
			\item[(a)] $\Hom_{\mathcal{D}/G}(X_t,\Sigma S_i)=0$ for any $X_t\in \mathcal{X}_t$ and $t\geq 0$. Consequently, every triangle of the form $S_i\rightarrow L\rightarrow X_t\rightarrow \Sigma S_i$ in $\mathcal{D}/G$ is split, where $X_t\in \mathcal{X}_t$.
			\item[(b)] $\mathcal{X}_t\subseteq \rep(Q)$ and $\Hom_{\mathcal{D}}(S_i,\theta X_{t})=0$ for any $X_{t}\in \mathcal{X}_{t}$ and $t> 0$. In particular, $\Hom_{\mathcal{D}/G}(S_i,\Sigma X_t)=\Hom_{\mathcal{D}}(S_i,\Sigma X_t)$,  and every triangle in the definition of $\mathcal{X}_t$ is the image of a triangle of $\mathcal{D}$ induced by a short exact sequence of $\rep(Q)$.
		\end{enumerate}
		We already know that claim $(a)$ holds for $t=0$. Assume that it holds for $t-1$.  For $X_t\in \mathcal{X}_t$, there is a triangle $X_{t-1}\rightarrow X_t\rightarrow S_i\rightarrow \Sigma X_{t-1}$ in $\mathcal{D}/G$. Applying $\Hom_{\mathcal{D}/G}(-,\Sigma S_i)$ to the triangle above yields an exact sequence
		\[
		0=\Hom_{\mathcal{D}/G}(S_i,\Sigma S_i)\rightarrow \Hom_{\mathcal{D}/G}(X_t,\Sigma S_i)\rightarrow \Hom_{\mathcal{D}/G}(X_{t-1},\Sigma S_i)=0,
		\]
		which implies that $\Hom_{\mathcal{D}/G}(X_t,\Sigma S_i)=0$.
		
		Let us prove the claim $(b)$. Let $X_1\in \mathcal{X}_1$ and there is a triangle
		\begin{eqnarray}\label{triangle-M1}
			S_j\stackrel{\tilde{f}}{\longrightarrow}  X_1\stackrel{\tilde{g}}{\longrightarrow} S_i\longrightarrow \Sigma S_j 
		\end{eqnarray}
		in $\mathcal{D}/G$. Note that $\Hom_{\mathcal{D}/G}(S_i,\Sigma S_j)=\Hom_{\mathcal{D}}(S_i,\Sigma S_j)$, which implies that the triangle \eqref{triangle-M1} is the image of a triangle
		\begin{equation}\label{tri:d-1}
			S_j\stackrel{f}{\longrightarrow}  X_1\stackrel{g}{\longrightarrow}  S_i\longrightarrow \Sigma S_j
		\end{equation}
		of $\mathcal{D}$. By Lemma \ref{lem:s.e.s-induce-tri}, $X_1\in \rep(Q)$ and the triangle \eqref{tri:d-1} is induced by a short exact sequence of $\rep(Q)$. Applying $\Hom_{\mathcal{D}}(S_i,\theta-)$ to \eqref{tri:d-1} yields an exact sequence
		\[
		0=\Hom_{\mathcal{D}}(S_i,\theta S_j)\rightarrow \Hom_{\mathcal{D}}(S_i,\theta X_1)\rightarrow \Hom_{\mathcal{D}}(S_i,\theta S_i)=0,
		\]
		which implies that $\Hom_{\mathcal{D}}(S_i,\theta X_1)=0$. Now assume that  claim $(b)$ is proved for $t-1$. For $X_{t-1}\in \mathcal{X}_{t-1}$, by $\Hom_{\mathcal{D}}(S_i,\theta X_{t-1})=0$, we compute \[\Hom_{\mathcal{D}/G}(S_i,\Sigma X_{t-1})=\Hom_{\mathcal{D}}(S_i,\Sigma X_{t-1})\oplus \Hom_{\mathcal{D}}(S_i,\theta X_{t-1})=\Hom_{\mathcal{D}}(S_i,\Sigma X_{t-1}).\]
		It follows that every triangle $X_{t-1}\rightarrow X_t\rightarrow S_i\rightarrow \Sigma X_{t-1}$ in $\mathcal{D}/G$ with $X_{t-1}\in \mathcal{X}_{t-1}$ is the image of a triangle of $\mathcal{D}$, which is induced by a short exact sequence of $\rep(Q)$ by Lemma \ref{lem:s.e.s-induce-tri}. Furthermore, $X_t\in \rep(Q)$. Applying $\Hom_{\mathcal{D}}(S_i,\theta-)$ to the triangle above yields $\Hom_{\mathcal{D}}(S_i,\theta X_t)=0$.
		
		By Lemma~\ref{lem:basis-go-M}, $\tilde{h}_{S_i}$ and $\tilde{h}_{S_j}$
		are linearly independent. It follows that every term $\tilde{u}_M$
		appearing in the computation of
		$(\operatorname{ad}\tilde{u}_{S_i})^{-c_{ij}+1}\tilde{u}_{S_j}$
		satisfies $S_i\not\simeq \Sigma M$. In particular, no term of the form
		$\frac{\tilde{h}_{S_i}}{\tilde{d}_{S_i}}$ occurs. Moreover, by (a) and (b), every triangle with indecomposable middle term
		that appears in this computation is induced by a short exact sequence in
		$\rep(Q)$. Thus the computation is reduced to one in $\rep(Q)$. Hence
		$(\operatorname{ad}\tilde{u}_{S_i})^{-c_{ij}+1}\tilde{u}_{S_j}=0$
		follows from \cite{R90-Hall-poly-rep-fin,PX00-Tri-Cat-Kac}. Applying the involution $\theta$ to $(\operatorname{ad} \tilde{u}_{{S_i}})^{-c_{ij}+1} \tilde{u}_{{S_j}}=0$  yields $(\operatorname{ad} \tilde{u}_{{S_{\bar i}}})^{-c_{ij}+1} \tilde{u}_{{S_{\bar j}}}=0$. This finishes the proof of \eqref{rel:serre-rel-<0}.
		
		The relations \eqref{rel:serre-rel->00} and \eqref{rel:serre-rel->01} can be verified in a similar way, and we omit the details.
	\end{proof}
	
	\subsection{The Ringel-Hall Lie algebras $\mathfrak{g}(\mathcal{D}/\Sigma^2)_{(q-1)}$ and $\mathfrak{g}(\mathcal{D}/G)_{(q-1)}$}
	
	In this subsection, we investigate the relation between
	$\mathfrak{g}(\mathcal{D}/\Sigma^2)_{(q-1)}$ and
	$\mathfrak{g}(\mathcal{D}/G)_{(q-1)}$. We retain the notation and
	conventions introduced at the end of
	Subsection~\ref{ss:2-periodic-(C,D)}.
	
	\begin{lemma}\label{lem:tilde-d=d}
		Let $X\in \mathcal{D}$ be an indecomposable object. Then $\tilde{d}_X=d_X$.
	\end{lemma}
	\begin{proof}
		Without loss of generality, we may assume that $X\in \rep(Q)$. It follows that \[\End_{\mathcal{D}/G}(X)=\Hom_{\mathcal{D}}(X,X)\oplus \Hom_{\mathcal{D}}(X,\Sigma \theta X) \ \text{and}\ \End_{\mathcal{D}/\Sigma^2}(X)=\Hom_{\mathcal{D}}(X,X).\] Furthermore, \[\rad \End_{\mathcal{D}/G}(X)=\rad \End_{\mathcal{D}}(X)\oplus \Hom_{\mathcal{D}}(X,\Sigma \theta X)\ \text{and}\ \rad \End_{\mathcal{D}/\Sigma^2}(X)=\rad \End_{\mathcal{D}}(X).\] Consequently, $\tilde{d}_X=d_X$.
	\end{proof}
	\begin{lemma}\label{lem:homom-cartan-subalgebra}
		There is  a group homomorphism $\psi:\go(\mathcal{D}/G)\rightarrow \go(\mathcal{D}/\Sigma^2)$ given by
		$\tilde{h}_X\mapsto h_X+h_{GX}$, $\forall X\in \mathcal{D}/G$. It further induces a homomorphism $\psi:\tilde{\mathfrak{h}}\rightarrow \mathfrak{h}$.
	\end{lemma}
	\begin{proof}
		According to Lemma \ref{lem:iso-groth-root-cat}, we have an isomorphism
		$\go(\mathcal{D})\cong \go(\mathcal{D}/\Sigma^2)$. Now $\psi$ is the composition of the homomorphism $\phi:\go(\mathcal{D}/G)\rightarrow \go(\mathcal{D})$ in Lemma \ref{lem:basis-go-M} with this isomorphism. The last statement is a direct consequence of Lemma \ref{lem:tilde-d=d}.
	\end{proof}

	\begin{lemma}\label{lem:hall-number-root-cat-zero}
		Let $L,M$ be indecomposable representations of $(Q,\mathbf{d})$. Then \[h_{S_i,GM}^L=h_{GS_i,GM}^L=h_{GM,S_i}^L=h_{GM,GS_i}^L=0\] for any $i\in Q_0$.
	\end{lemma}
	\begin{proof}
		Let us prove $h_{S_i,GM}^L=0$, the others are similar. Assume that $h_{S_i,GM}^L\neq 0$. There is a triangle 
		\begin{equation}\label{tri-DSigma^2}
			GM\longrightarrow L\longrightarrow S_i\longrightarrow \Sigma GM
		\end{equation}
		in $\mathcal{D}/\Sigma^2$. Note that $\Hom_{\mathcal{D}/\Sigma^2}(L,S_i)=\Hom_\mathcal{D}(L, S_{i})$. It follows that the triangle \eqref{tri-DSigma^2} is the image of a triangle in $\mathcal{D}$ under the canonical projection $\pi:\mathcal{D}\rightarrow \mathcal{D}/\Sigma^2$. Therefore, $\hat{L}=\hat{GM}+\hat{S}_i=-\hat{\theta M}+\hat{S}_i$ in $\go(\mathcal{D})=\go(Q)$, which contradicts to the condition that $L,M$ are indecomposable representations.
	\end{proof}
	
	
	
	\begin{lemma} \label{lem:equ-Hall-num}
		Let $L$ and $M$ be indecomposable representations of $(Q,\mathbf{d})$. For any vertex $i\in Q_0$, the following equalities of Hall numbers hold in $\mathbb{Z}/(q-1)\mathbb{Z}$:
		\begin{align} 
			g_{{M},{S_i}}
			^{L}
			&\equiv h_{M,{S_i}}^L+h_{M,G{S_i}}^L+h_{GM,{S_i}}^L
			+h_{GM,G{S_i}}^L \ \  \pmod{q-1} \label{eq:hall-number-1},  \\
			g_{S_i,M}
			^{{L}} &\equiv h_{{S_i},M}^L+h_{{S_i},GM}^L+h_{G{S_i},M}^L+h_{G{S_i},GM}^L\ \  \pmod{q-1}.\label{eq:hall-number-2}
		\end{align}
		
	\end{lemma}
	\begin{proof}
		We give the proof of \eqref{eq:hall-number-2}, while the proof of \eqref{eq:hall-number-1} is similar. According to Lemma \ref{lem:hall-number-root-cat-zero}, it remains to show $g_{S_i,M}^L\equiv h_{S_i,M}^L+h_{GS_i,M}^L$.
		Recall that \[\Hom_{\mathcal{D}/G}(L,S_i)=\Hom_{\mathcal{D}}(L,S_i)\oplus \Hom_{\mathcal{D}}(L,GS_i).\] It follows that every morphism in $\Hom_{\mathcal{D}/G}(L,S_i)$ can be uniquely expressed as $f+g$, where $f\in \Hom_{\mathcal{D}}(L,S_i)$ and $g\in \Hom_{\mathcal{D}}(L,GS_i)$. Then there is a triangle in $\mathcal{D}/G$:
		\begin{equation}\label{tri:mapping-cone}
			M\longrightarrow L\xrightarrow{f+g} S_i\longrightarrow \Sigma M. 
		\end{equation}
		Recall that $\rep(Q)\simeq \mod H$ for some hereditary algebra $H$, and applying the functor $\pi_\theta$ defined in Lemma~\ref{lem:pi-theta-fully-fathful} to the triangle \eqref{tri:mapping-cone}, we obtain a triangle in $\mathcal{D}_\theta(H)$
		\begin{equation}\label{tri-DA-fg}
			\pi_\theta M\longrightarrow \pi_\theta L\xrightarrow{\pi_\theta(f+g)} \pi_\theta S_i\longrightarrow \pi_\theta \Sigma M.
		\end{equation}
		Let $P_2\stackrel{\sigma}{\rightarrowtail} P_1\twoheadrightarrow L$ be a minimial projective resolution of $L$. Then $\pi_\theta L$ can be expressed as the following $\theta$-complex in $\rep(Q,\mathbf{d})$:
		\[
		\cdots\rightarrow \theta P_1\oplus P_2\xrightarrow{d^{-1}=\begin{pmatrix}0&\sigma \\ 0&0\end{pmatrix}} P_1\oplus \theta P_2\xrightarrow{d^0=\begin{pmatrix}0&\theta \sigma \\ 0&0\end{pmatrix}}\theta P_1\oplus P_2\rightarrow \cdots.
		\]
		Similarly, $\pi_\theta S_i$ is isomorphic to the following $\theta$-complex:
		\[
		\cdots\rightarrow S_{\bar{i}}\xrightarrow{\partial^{-1}=0}S_i\xrightarrow{\partial^0=0} S_{\bar{i}}\rightarrow \cdots.
		\]
		As a consequence, the morphisms among the last three terms of the triangle \eqref{tri-DA-fg} can be expressed as chain maps:
		\[
		\xymatrix{ 
			\cdots &
			{\theta P_1\oplus P_2}
			\ar[d]_{  \begin{pmatrix}     
					\theta f_1 & f_2 
			\end{pmatrix}} 
			\ar[rr]^{
				\begin{pmatrix}     
					0 & \sigma \\
					0 &0
			\end{pmatrix}}  
			& & {P_1\oplus\theta P_2}
			\ar[d]^{
				\begin{pmatrix}     
					f_1 & \theta f_2 
			\end{pmatrix}}\ar[rr]^{\begin{pmatrix}     
					0 & \theta\sigma \\
					0 &0
			\end{pmatrix}}& &\theta P_1\oplus P_2 \ar[d]^{\begin{pmatrix}     
					\theta f_1 &  f_2 
			\end{pmatrix}} &\cdots \\
			\cdots &
			{S_{\bar i}}\ar[d]\ar[rr]^0 &&  {S_i}\ar[d]\ar[rr]^0& &S_{\bar{i}}\ar[d] &\cdots \\
			\cdots &
			{P_1\oplus \theta P_2\oplus S_{\bar i}} 
			\ar[rr]^{
				\psi^{-1}
			} 
			& & \theta P_1\oplus P_2\oplus S_i\ar[rr]^{\psi^0} &&{P_1\oplus \theta P_2\oplus S_{\bar i}}& \cdots,
		}
		\]
		where 
		$\psi^{-1}=\begin{pmatrix}         0 & -\theta \sigma & 0  \\         0 & 0 & 0\\          f_1 & \theta f_2 & 0     \end{pmatrix}$ and $\psi^0=\begin{pmatrix}         0 & -\sigma & 0  \\         0 & 0 & 0\\          \theta f_1 &  f_2 & 0     \end{pmatrix}$.
		
		Now assume that $f\neq 0$, equivalently, $f_1\neq 0$. Consequently, $f_1$ is surjective, since $S_i$ is a simple representation. A direct computation shows that $\operatorname{im} \psi^{-1}=\operatorname{im} \theta \sigma \oplus S_i$ and $\ker  \psi^0=\ker \theta f_1\oplus S_i$ by noticing that $\sigma$ is injective. Hence $H^0(\pi_\theta \Sigma M)\cong \ker\theta f_1/\operatorname{im} \theta\sigma$. On the other hand, by~\eqref{eq:pi-theta-object}, we have
		$H^0(\pi_\theta \Sigma M)\simeq H^1(\pi_\theta M)\simeq \theta M$. In particular, the mapping cone of $\pi_\theta(f+g)$ is uniquely determined by $f$. In other words, for any $h\in \Hom_\mathcal{D}(L,GS_i)$, the mapping cone of $f+h$ is also $\Sigma M$.

		If there is a nonzero morphism $f\in\Hom_{\mathcal{D}}(L,S_i)$ such that $f^\ast\in \Hom_{\mathcal{D}/G}(L,S_i)^\ast_{\Sigma M}$, then there exists a triangle
		\[
		M\longrightarrow L\longrightarrow S_i\longrightarrow \Sigma M
		\]
		in $\mathcal{D}$. It follows that $\hat{L}=\hat{M}+\hat{S}_i\in \go(\mathcal{D})$. Hence $\hat{L}\neq \hat{M}+\hat{GS_i}=\hat{M}-\hat{S}_{\bar{i}}$. Consequently, the mapping cone of any morphism $g\in \Hom_{D}(L, GS_i)$ is not isomorphic to $\Sigma M$. Putting all of these together, we conclude that
		\[
		\Hom_{\mathcal{D}/G}(L,S_i)_{\Sigma M}=\Hom_{\mathcal{D}}(L,S_i)_{\Sigma M}\times \Hom_{\mathcal{D}}(L,GS_i).
		\]
		On the other hand, $\Hom_{\mathcal{D}/\Sigma^2}(L,S_i)=\Hom_{\mathcal{D}}(L,S_i)$ and $\Hom_{\mathcal{D}/\Sigma^2}(L,GS_i)=\Hom_{\mathcal{D}}(L,GS_i)$. It follows that 
		\[\Hom_{\mathcal{D}/\Sigma^2}(L,S_i)_{\Sigma M}=\Hom_{\mathcal{D}}(L,S_i)_{\Sigma M}\ \text{and}\ h_{GS_i,M}^L=0.
		\]
		Note that $\End_{\mathcal{D}/G}(S_i)=\End_{\mathcal{D}/\Sigma^2}(S_i)=\End_{\mathcal{D}}(S_i)$ is a division algebra. It follows that the action of $\Aut_{\mathcal{D}/G}(S_i)=\Aut_{\mathcal{D}/\Sigma^2}(S_i)$ on $\Hom_{\mathcal{D}/G}(L,S_i)_{\Sigma M}$ and $\Hom_{\mathcal{D}/\Sigma^2}(L,S_i)_{\Sigma M}$ are free. Hence
		\begin{eqnarray*}
			g_{S_i,M}^L&=&\frac{|\Hom_{\mathcal{D}/G}(L,S_i)_{\Sigma M}|}{|\Aut_{\mathcal{D}/G}(S_i)|}=\frac{|\Hom_{\mathcal{D}}(L,S_i)_{\Sigma M}|\cdot |\Hom_{\mathcal{D}}(L,GS_i)|}{|\Aut_{\mathcal{D}}(S_i)|}\\
			&=&\frac{|\Hom_{\mathcal{D}/\Sigma^2}(L,S_i)_{\Sigma M}|}{\Aut_{\mathcal{D}/\Sigma^2}(S_1)}\times q^{\dim \Hom_{\mathcal{D}}(L,GS_i)}\\
			&=&h_{S_i,M}^L\times q^{\dim \Hom_{\mathcal{D}}(L,GS_i)}.
		\end{eqnarray*}
		
		Now assume that there does not exist a morphism $f\in \Hom_{\mathcal{D}}(L,S_i)$ such that $f^\ast\in \Hom_{\mathcal{D}/G}(L,S_i)^\ast_{\Sigma M}$. It follows that $h_{S_i,M}^L=0$ and
		\begin{eqnarray*}
			g_{S_i,M}^L&=&\frac{|\Hom_{\mathcal{D}}(L,GS_i)_{\Sigma M}|}{|\Aut_{\mathcal{D}}(GS_i)|}\\
			&=&\frac{|\Hom_{\mathcal{D}/\Sigma^2}(L,GS_i)_{\Sigma M}|}{|\Aut_{\mathcal{D}}(GS_i)|}\\
			&=&h_{GS_i,M}^L.
		\end{eqnarray*}
		The proof is completed.
	\end{proof}

	Let $\mathfrak{c}(\mathcal{D}/G)_{(q-1)}$ be the Lie subaglebra of $\mathfrak{g}(\mathcal{D}/G)_{(g-1)}$ generated by \[\{\tilde{u}_{S_i},\frac{\tilde{h}_{S_i}}{\tilde{d}_{S_i}}\mid i=1,\ldots, n,\bar{1},\ldots, \bar{n}\},\]
	and $\mathfrak{c}(\mathcal{D}/\Sigma^2)_{(q-1)}$ be the Lie subaglebra of $\mathfrak{g}(\mathcal{D}/\Sigma^2)_{(g-1)}$ generated by \[\{{u}_{S_i}, u_{\Sigma S_i},\frac{{h}_{S_i}}{{d}_{S_i}}\mid i=1,\ldots, n,\bar{1},\ldots, \bar{n}\}.\]
	According to Lemma \ref{lem:homom-cartan-subalgebra}, the following is a well-defined $\mathbb{Z}$-linear map:
	\begin{align*}
		\varphi: & \mathfrak{g}(\mathcal{D}/G)_{(q-1)}  \longrightarrow \mathfrak{g}(\mathcal{D}/\Sigma^2)_{(q-1)}. \\
		&\tilde{h}_{X} \longmapsto h_X+h_{GX} \quad (X\in \mathcal{D}/G),\\
		&\tilde{u}_{Y} \longmapsto u_Y+u_{GY} \quad (Y\in \operatorname{ind}\mathcal{D}/G).
	\end{align*}

	\begin{lemma}\label{lem:property-varphi}
		The map $\varphi$ satisfies the following properties:
		\begin{enumerate}
			\item $\varphi([\tilde{h}_X,\tilde{h}_Y])=[\varphi(\tilde{h}_X),\varphi(\tilde{h}_Y)]$ for any $X,Y\in \mathcal{D}/G$;
			\item $\varphi([\tilde{h}_X,\tilde{u}_Y])=[\varphi(\tilde{h}_X),\varphi(\tilde{u}_Y)]$ for any $X\in \mathcal{D}/G$ and $Y\in \operatorname{ind} \mathcal{D}/G$;
			\item $\varphi([\tilde{u}_X,\tilde{u}_{S_i}])=[\varphi(\tilde{u}_X),\varphi(\tilde{u}_{S_i})]$ for any $X\in \operatorname{ind} \mathcal{D}/G$ and $i\in Q_0$.
		\end{enumerate}
	\end{lemma}
	\begin{proof}
		The statement $(1)$ is evident by definition. For the statement $(2)$, by the definitions of Lie brackets and $\varphi$, we have 
		\begin{align*}
			\varphi([\tilde{h}_{X} ,\tilde{u}_{Y}])&=\varphi((\tilde{h}_{{X}},\tilde{h}_{{Y}})_{\mathcal{D}/G} \cdot \tilde{u}_{{Y}})= (\tilde{h}_{{X}},\tilde{h}_{{Y}})_{\mathcal{D}/G} \cdot (u_Y+u_{GY}), \\
			[\varphi(\tilde{h}_{{X}}),\varphi(\tilde{u}_{{Y}})]&=[h_X+h_{GX},u_Y+u_{GY}]\\
			&=[(h_X,h_Y)_{\mathcal{D}/\Sigma^2}+(h_{GX},h_Y)_{\mathcal{D}/\Sigma^2}]\cdot u_Y 
			+[(h_X,h_{GY})_{\mathcal{D}/\Sigma^2}\\
			&\quad+(h_{GX},h_{GY})_{\mathcal{D}/\Sigma^2}]\cdot u_{GY}.
		\end{align*}
		It remains to show that 
		\[
		(h_X,h_Y)_{\mathcal{D}/\Sigma^2}+(h_{GX},h_Y)_{\mathcal{D}/\Sigma^2} 
		=(\tilde{h}_{{X}},\tilde{h}_{{Y}})_{\mathcal{D}/G}
		=
		(h_X,h_{GY})_{\mathcal{D}/\Sigma^2}+(h_{GX},h_{GY})_{\mathcal{D}/\Sigma^2}.
		\]
		By linearity, it suffices to assume that $X=S_i$ and $Y=S_j$ for $1\leq i\leq j\leq n$, which can be verified by a direct computation.
		
		Let us prove the statement $(3)$. Without loss of generality, we may assume that $X$ is an indecomposable representation of $(Q,\mathbf{d})$. Note that $X\cong \Sigma S_i$ in $\mathcal{D}/G$ if and only if $X\cong S_{\bar{i}}$ in $\rep(Q)$.
		
		Let us first assume that $X\not\cong \Sigma S_i$ in $\mathcal{D}/G$.
		On the one hand, we have
		\begin{align*}
			\varphi([ \tilde{u}_{{X}}, \tilde{u}_{{S_i}}])=& \varphi(\sum_{{L}\in \operatorname{ind}\mathcal{D}/G}  \tilde{\gamma}_{{X}, {S_i}}^{{L}} \tilde{u}_{{L}})
			= \sum_{L\in\operatorname{ind}(Q)}\tilde{\gamma}_{{X},{S_i}}^{{L}}(u_L+u_{GL}).
		\end{align*}
		On the other hand, the condition $X\not\cong \Sigma S_i$ in $\mathcal{D}/G$ implies that $X\not\cong \Sigma S_i,X\not\cong\Sigma GS_i$, $GX\not\cong\Sigma S_i$ and $GX\not\cong \Sigma GS_i$ in $\mathcal{D}/\Sigma^2$. It follows that
		\begin{align*}
			[\varphi( \tilde{u}_{{X}}),\varphi(\tilde{u}_{{S_i}})]
			=&
			[u_X+u_{GX},u_{S_i}+u_{G{S_i}}]\\
			=&\sum_{L\in\operatorname{ind}\mathcal{D}/\Sigma^2}(\gamma_{X,{S_i}}^L+\gamma_{X,G{S_i}}^L+\gamma_{GX,{S_i}}^L+\gamma_{GX,G{S_i}}^L)u_L \\
			=& \sum_{L\in\operatorname{ind}(Q)}(\gamma_{X,{S_i}}^L+\gamma_{X,G{S_i}}^L+\gamma_{GX,{S_i}}^L+\gamma_{GX,G{S_i}}^L)u_L 
			\\
			&+\sum_{L\in\operatorname{ind}(Q)}(\gamma_{X,{S_i}}^{GL}+\gamma_{X,G{S_i}}^{GL}+\gamma_{GX,{S_i}}^{GL}+\gamma_{GX,G{S_i}}^{GL})u_{GL}
			\\
			=&\sum_{L\in\operatorname{ind}(Q)}(\gamma_{X,{S_i}}^L+\gamma_{X,G{S_i}}^L+\gamma_{GX,{S_i}}^L+\gamma_{GX,G{S_i}}^L)(u_L+u_{GL}). 
		\end{align*}
		We conclude that  $\varphi([ \tilde{u}_{{X}}, \tilde{u}_{{S_i}}])=[\varphi( \tilde{u}_{{X}}),\varphi(\tilde{u}_{{S_i}})]$ by Lemma \ref{lem:equ-Hall-num}.
		
		Now assume that $X\cong \Sigma S_i$ in $\mathcal{D}/G$, hence we may assume that $X= S_{\bar{i}}$.
		Note that 
		\[\Hom_{\mathcal{D}/G}(S_i,\Sigma S_{\bar{i}})=\Hom_{\mathcal{D}}(S_i,\Sigma S_{\bar{i}})\oplus \Hom_{\mathcal{D}}(S_i,S_i)=\Hom_{\mathcal{D}}(S_i,S_i),\]
		and
		\[\Hom_{\mathcal{D}/G}(S_{\bar{i}},\Sigma S_{i})=\Hom_{\mathcal{D}}(S_{\bar{i}},\Sigma S_i)\oplus \Hom_{\mathcal{D}}(S_{\bar{i}},S_{\bar{i}})=\Hom_{\mathcal{D}}(S_{\bar{i}},S_{\bar{i}}).
		\]
		It follows that $\tilde{\gamma}_{S_{\bar{i}},S_i}^L=0$ for any indecomposable object $L$ of $\mathcal{D}/G$. Hence, 
		$[\tilde{u}_{S_{\bar{i}}},\tilde{u}_{S_i}]=-\frac{h_{S_{\bar{i}}}}{d_{S_{\bar{i}}}}$, and 
		\[
		\varphi([\tilde{u}_{S_{\bar{i}}},\tilde{u}_{S_i}])=-\frac{h_{S_{\bar{i}}}-h_{S_i}}{d_{S_i}}.
		\]
		On the other hand,
		\begin{eqnarray*}
			[\varphi( \tilde{u}_{{S_{\bar{i}}}}),\varphi(\tilde{u}_{{S_i}})]&=&[u_{S_{\bar{i}}}+u_{G S_{\bar{i}}},u_{S_i}+u_{G S_i}]\\
			&=&[u_{S_{\bar{i}}}+u_{\Sigma S_i},u_{S_i}+u_{\Sigma S_{\bar{i}}}]\\
			&=&[u_{S_{\bar{i}}},u_{\Sigma S_{\bar{i}}}]+[u_{\Sigma S_i},u_{S_i}]\\
			&=&-\frac{h_{S_{\bar{i}}}}{d_{S_{\bar{i}}}}-\frac{h_{\Sigma S_i}}{d_{\Sigma S_i}}.
		\end{eqnarray*}
		Noticing that $h_{\Sigma S_i}=-h_{S_i}$ and $d_{S_{\bar{i}}}=d_{\Sigma S_i}=d_{S_i}$, we conclude that \[\varphi([\tilde{u}_{S_{\bar{i}}},\tilde{u}_{S_i}])=[\varphi( \tilde{u}_{{S_{\bar{i}}}}),\varphi(\tilde{u}_{{S_i}})].\]
		This completes the proof.
	\end{proof}
	\begin{proposition}\label{pro:homom-Ringel-Hall-algs}
		The restriction of $\varphi$
		on $\mathfrak{c}(\mathcal{D}/G)_{(q-1)} $ induces a Lie algebra homomorphism
		\[
		\bar{\varphi}: \mathfrak{c}(\mathcal{D}/G)_{(q-1)}\longrightarrow \mathfrak{c}(\mathcal{D}/\Sigma^2)_{(q-1)}.
		\]
	\end{proposition}
	\begin{proof}
		By Lemma \ref{lem:property-varphi} and the Jacobi identity, we deduce $\operatorname{im} \varphi|_{\mathfrak{c}(\mathcal{D}/G)_{(q-1)}}\subseteq \mathfrak{c}(\mathcal{D}/\Sigma^2)_{(q-1)}$. Hence, the restriction of $\varphi$ yields a linear map $\bar{\varphi}:\mathfrak{c}(\mathcal{D}/G)_{(q-1)}\rightarrow \mathfrak{c}(\mathcal{D}/\Sigma^2)_{(q-1)}$. Again by Lemma \ref{lem:property-varphi}, the linearity of ${\varphi}$ and the Jacobi identity, we conclude that $\bar{\varphi}$ is a homomorphism of Lie algebras.
	\end{proof}
	

	\subsection{GIM algebras via Ringel--Hall Lie algebras}\label{ss:main result GIM}
	We consider the integral Ringel--Hall Lie algebra of $\mathcal{D}/G$ as in Subsection \ref{ss:intergral-R-H-Lie-alg}.
	Set
	\[
	\begin{aligned}
		\Omega = \{ \mathbb{K} \mid \mathbb{F} \subseteq \mathbb{K} \subseteq \ov{\mathbb{F}} \text{ is a finite field extension and conservative } \\
		\text{for all simple representations in } \rep_{\mathbb{F}}(Q,\mathbf{d}) \}.
	\end{aligned}
	\]
	
	For each $\mathbb{K}\in \Omega$, denote by $S_i^{\mathbb{K}}$ the simple representation in $\rep_{\mathbb{K}}(Q,\mathbf{d})$ associated to vertex $i$.
	Let $\mathfrak{g}((\mathcal{D}/G)^\mathbb{K})_{(|\mathbb{K}|-1)}$ be the Ringel-Hall Lie algebra of the triangulated orbit category $(\mathcal{D}/G)^\mathbb{K}:=\mathcal{D}^b(\rep_{\mathbb{K}}(Q,\mathbf{d}))/G$.
	We consider the direct product of the Lie algebras \[\prod_{\mathbb{K}\in \Omega} \mathfrak{g}((\mathcal{D}/G)^\mathbb{K})_{(|\mathbb{K}|-1)},\] which has a natural $\mathbb{Z}$-Lie algebraic structure, and denote $\mathscr{LC}(\mathcal{D}/G)$ as its Lie subalgebra generated by 
	\[
	\tilde{\bf u}_{S_i}:=(\tilde{u}_{{S_i}^\mathbb{K}})_{\mathbb{K}\in\Omega},\quad
	\tilde{\bf u}_{S_{\bar{i}}}:=
	(\tilde{ u}_{ {S_{\bar{i}}}^\mathbb{K}})_{\mathbb{K}\in\Omega},\quad
	\tilde{\bf h}_{S_i}:=(\tilde{h}_{{S_i}^\mathbb{K}})_{\mathbb{K}\in\Omega}
	,i=1,\dots,n.  
	\]
	The following is a direct consequence of Proposition \ref{pro:homom-Ringel-Hall-algs}.
	\begin{corollary}
		There is a Lie algebra homomorphism $ \Phi:\mathscr{LC}(\mathcal{D}/G)\rightarrow \mathscr{LC}(\mathcal{D}/\Sigma^2)$ such that $\Phi(\tilde{\bf u}_{S_i})={\bf u}_{S_i}+{\bf u}_{\Sigma S_{\bar{i}}}$
		and $\Phi(\tilde{\bf h}_{S_i})={\bf h}_{S_i}-{\bf h}_{S_{\bar{i}}}$ for $ i=1,\ldots,n,\bar{1},\ldots,\bar{n}$.
	\end{corollary}
	
	Recall that, by Remark \ref{rem:twisted-embedding}, there exists an embedding $\iota':{\rm gim}(C)\rightarrow {\rm gcm}(A(C))$. On the other hand, by Lemmas \ref{lem:iso-groth-root-cat}, \ref{lem:Ringel-Euler-form-(Q,d)} and Theorem \ref{thm:PX-realization}, there is an isomorphism of Lie algebras $\Xi:{\rm gcm}(A(C))\rightarrow \mathscr{LC}(\mathcal{D}/\Sigma^2)\otimes_{\mathbb{Z}} \mathbb{C}$.
	Now we are in a position to state the main result of this paper.
	\begin{theorem}\label{thm:GIM-via-RH-Lie}
		The GIM algebra ${\rm gim}(C)$ is isomorphic to the integral Ringel--Hall Lie algebra $\mathscr{LC}(\mathcal{D}/G)\otimes_{\mathbb{Z}} \mathbb{C}$, where the isomorphism  $\Psi:{\rm gim}(C)\cong\mathscr{LC}(\mathcal{D}/G)\otimes_{\mathbb{Z}} \mathbb{C}$ is given by
		\[
		\tilde{e}_i\mapsto \tilde{\bf u}_{S_i},\quad \tilde{f}_i\mapsto -\tilde{\bf u}_{S_{\bar{i}}}, \quad\tilde{h}_i\mapsto \frac{\tilde{\bf h}_{S_i}}{\tilde{d}_{S_i}},\quad 1\leq i\leq n.
		\]
		Furthermore, $\Psi$ fits into the following commutative diagram
		\[
		\xymatrix{{\rm gim}(C)\ar[rr]^{\iota'}\ar[d]^{\Psi} &&{\rm gcm}(A(C))\ar[d]^{\Xi}\\
			\mathscr{LC}(\mathcal{D}/G)\otimes_{\mathbb{Z}} \mathbb{C}\ar[rr]^{\Phi\otimes \id}&&\mathscr{LC}(\mathcal{D}/\Sigma^2)\otimes_{\mathbb{Z}} \mathbb{C}.
		}
		\]
	\end{theorem}
	\begin{proof}
		We first prove that $\Psi$ is a homomorphism of Lie algebras. It suffices to verify that $\Psi$ preserves the defining relations \eqref{rel:gim-1}--\eqref{rel:gim-5} of ${\rm gim}(C)$. The relation \eqref{rel:gim-1} is evident by definition, while relations \eqref{rel:gim-4} and \eqref{rel:gim-5} follow from Proposition \ref{pro:Serre-relation}.
		For any $1\leq i,j\leq n$,
		\[
		[\Psi(\tilde{h}_i),\Psi(\tilde{e}_j)]=[\frac{\tilde{\bf h}_{S_i}}{\tilde{d}_{S_i}},\tilde{\bf u}_{S_j}]
		=\frac{(\tilde{\bf h}_{S_i},\tilde{\bf h}_{S_j})}{\tilde{d}_{S_i}}\tilde{\bf u}_{S_j}
		=c_{ij}\tilde{\bf u}_{S_j}=c_{ij}\Psi(\tilde{e}_j),
		\]
		where the third equality follows from Lemma \ref{lem:Euler-form-D/G}. Similarly, by noticing that $\tilde{\bf h}_{S_{\bar{j}}}=-\tilde{\bf h}_{S_j}$, we obtain
		\[
		[\Psi(\tilde{h}_i),\Psi(\tilde{f}_j)]=[\frac{\tilde{\bf h}_{S_i}}{\tilde{d}_{S_i}},-\tilde{\bf u}_{S_{\bar{j}}}]
		=-\frac{(\tilde{\bf h}_{S_i},\tilde{\bf h}_{S_{\bar{j}}})}{\tilde{d}_{S_i}}\tilde{\bf u}_{S_{\bar{j}}}=\frac{(\tilde{\bf h}_{S_i},\tilde{\bf h}_{S_{{j}}})}{\tilde{d}_{S_i}}\tilde{\bf u}_{S_{\bar{j}}}
		=c_{ij}\tilde{\bf u}_{S_{\bar{j}}}=-c_{ij}\Psi(\tilde{f}_{\bar{j}}).
		\]
		This completes the verification of \eqref{rel:gim-2}.
		
		Note that $S_{\bar{i}}\cong \Sigma S_i$ in $\mathcal{D}/G$ and 
		\[
		\Hom_{\mathcal{D}/G}(S_i,\Sigma S_{\bar{i}})=\Hom_{\mathcal{D}}(S_i,\Sigma S_{\bar{i}})\oplus \Hom_{\mathcal{D}}(S_i,S_i)=\Hom_{\mathcal{D}}(S_i,S_i),
		\]
		\[
		\Hom_{\mathcal{D}/G}(S_{\bar{i}},\Sigma S_{{i}})=\Hom_{\mathcal{D}}(S_{\bar{i}},\Sigma S_{{i}})\oplus \Hom_{\mathcal{D}}(S_{\bar{i}},S_{\bar{i}})=\Hom_{\mathcal{D}}(S_{\bar{i}},S_{\bar{i}}).
		\]
		In particular, $\tilde{\gamma}_{S_i,S_{\bar{i}}}^L=0$ for any indecomposable object $L\in \mathcal{D}/G$. Hence, \[[\Psi(\tilde{e}_i),\Psi(\tilde{f}_i)]=[\tilde{\bf u}_{S_i},-\tilde{\bf u}_{S_{\bar{i}}}]=\frac{\tilde{\bf h}_{S_i}}{\tilde{d}_{S_i}}=\Psi(h_i).\] This completes the verification for \eqref{rel:gim-3} and that $\Psi$ is a homomorphism of Lie algebras.
		
		By the definition of $\mathscr{LC}(\mathcal{D}/G)\otimes_{\mathbb{Z}} \mathbb{C}$, we know that $\Psi$ is surjective. A direct computation shows that $\Xi\circ \iota'=(\Phi\otimes\id)\circ \Psi$. Since $\Xi\circ \iota'$ is injective, we conclude that $\Psi$ is injective, hence an isomorphism.
	\end{proof}

	\section{Elliptic Lie Algebras and Ringel--Hall Lie Algebras}
	\label{sec:ell-Lie}
	
	In this section, we apply the framework developed in Sections \ref{s:2-periodic-M} and \ref{s:GIM-RH-Lie} to investigate the  elliptic Lie algebras of types $D_4^{(1,1)}$, $E_6^{(1,1)}$, $E_7^{(1,1)}$, and $E_8^{(1,1)}$. Moving forward, we shall denote these four types uniformly by $X_l^{(1,1)}$.
	
	Note that each elliptic Lie algebra of the above types can be realized as a quotient of a suitable symmetrizable GIM algebra. Thus, starting from the corresponding elliptic Cartan matrices, we construct quivers
	$Q(X_l^{(1,1)})$ equipped with an involution $\theta$, in the same way as in
	Section~\ref{s:GIM-RH-Lie}. We then impose suitable relations on these quivers, preserved by $\theta$, so that $\theta$ induces an involution on the resulting quotient algebra
	$A(X_l^{(1,1)})$. In this way, we obtain 2-periodic triangulated categories whose Ringel--Hall Lie algebras are closed related to these simply-laced elliptic Lie algebras.

	\subsection{Elliptic Lie Algebras}\label{ss:Elliptic Lie Algebra}
	This subsection reviews the definitions and basic notions concerning the elliptic Lie algebras of types
	$D_4^{(1,1)}$, $E_6^{(1,1)}$, $E_7^{(1,1)}$, and $E_8^{(1,1)}$;
	see \cite{SY00}. Their elliptic Dynkin diagrams are shown in Figures~\ref{fig:dynkin-D4-11}--\ref{fig:dynkin-E8-11}.

	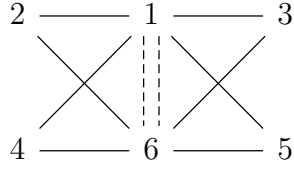
\begin{figure}[htbp]
		\centering
		\begin{tikzcd}[row sep=0.6cm, column sep=0.6cm]
			&& 2 && 1 && 3 \\
			\\
			&& 4 && 6 && 5
			\arrow[no head, from=1-3, to=1-5]
			\arrow[no head, from=1-3, to=3-5]
			\arrow[no head, from=1-5, to=1-7]
			\arrow[no head, from=1-5, to=3-3]
			\arrow[shift right, dashed, no head, from=1-5, to=3-5]
			\arrow[shift left, dashed, no head, from=1-5, to=3-5]
			\arrow[no head, from=1-5, to=3-7]
			\arrow[no head, from=1-7, to=3-5]
			\arrow[no head, from=3-3, to=3-5]
			\arrow[no head, from=3-5, to=3-7]
		\end{tikzcd}
		\caption{The elliptic Dynkin diagram of type $D_4^{(1,1)}$.}
		\label{fig:dynkin-D4-11}
	\end{figure}
	
	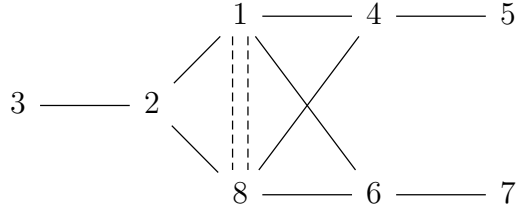
\begin{figure}[htbp]
		\centering
		\begin{tikzcd}[row sep=0.6cm, column sep=0.6cm]
			&&& 1 && 4 && 5 \\
			3 && 2 \\
			&&& 8 && 6 && 7
			\arrow[no head, from=1-4, to=1-6]
			\arrow[shift left, dashed, no head, from=1-4, to=3-4]
			\arrow[shift right, dashed, no head, from=1-4, to=3-4]
			\arrow[no head, from=1-4, to=3-6]
			\arrow[no head, from=1-6, to=1-8]
			\arrow[no head, from=2-1, to=2-3]
			\arrow[no head, from=2-3, to=1-4]
			\arrow[no head, from=2-3, to=3-4]
			\arrow[no head, from=3-4, to=1-6]
			\arrow[no head, from=3-4, to=3-6]
			\arrow[no head, from=3-6, to=3-8]
		\end{tikzcd}
		\caption{The elliptic Dynkin diagram of type $E_6^{(1,1)}$.}
		\label{fig:dynkin-E6-11}
	\end{figure}
	
	\begin{figure}[htbp]
		\centering
		\begin{tikzcd}[row sep=0.6cm, column sep=0.6cm]
			&& 1 && 3 && 4 && 5 \\
			2 \\
			&& 9 && 6 && 7 && 8
			\arrow[no head, from=1-3, to=1-5]
			\arrow[shift left, dashed, no head, from=1-3, to=3-3]
			\arrow[shift right, dashed, no head, from=1-3, to=3-3]
			\arrow[no head, from=1-3, to=3-5]
			\arrow[no head, from=1-5, to=1-7]
			\arrow[no head, from=1-7, to=1-9]
			\arrow[no head, from=2-1, to=1-3]
			\arrow[no head, from=2-1, to=3-3]
			\arrow[no head, from=3-3, to=1-5]
			\arrow[no head, from=3-3, to=3-5]
			\arrow[no head, from=3-5, to=3-7]
			\arrow[no head, from=3-7, to=3-9]
		\end{tikzcd}
		\caption{The elliptic Dynkin diagram of type $E_7^{(1,1)}$.}
		\label{fig:dynkin-E7-11}
	\end{figure}
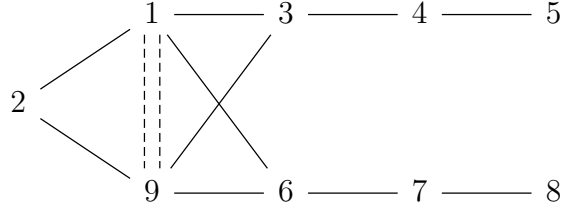
	
	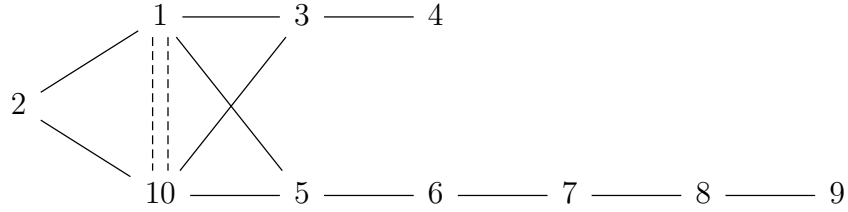
\begin{figure}[htbp]
		\centering
		\begin{tikzcd}[row sep=0.6cm, column sep=0.6cm]
			&& 1 && 3 && 4 \\
			2 \\
			&& 10 && 5 && 6 && 7 && 8 && 9
			\arrow[no head, from=1-3, to=1-5]
			\arrow[shift left, dashed, no head, from=1-3, to=3-3]
			\arrow[shift right, dashed, no head, from=1-3, to=3-3]
			\arrow[no head, from=1-3, to=3-5]
			\arrow[no head, from=1-5, to=1-7]
			\arrow[no head, from=2-1, to=1-3]
			\arrow[no head, from=2-1, to=3-3]
			\arrow[no head, from=3-3, to=1-5]
			\arrow[no head, from=3-3, to=3-5]
			\arrow[no head, from=3-5, to=3-7]
			\arrow[no head, from=3-7, to=3-9]
			\arrow[no head, from=3-9, to=3-11]
			\arrow[no head, from=3-11, to=3-13]
		\end{tikzcd}
		\caption{The elliptic Dynkin diagram of type $E_8^{(1,1)}$.}
		\label{fig:dynkin-E8-11}
	\end{figure}

	For each elliptic Dynkin diagram $X_{l}^{(1,1)}$, let $I=\{1,\dots, n\}$ be its vertex set, and let
	$V := V(X_{l}^{(1,1)})$
	be a $\mathbb{Q}$-vector space with basis
	$\Pi=\{\alpha_1,\dots,\alpha_{n}\}$. We define a symmetric bilinear form
	$\omega: V\times V\longrightarrow \mathbb{Q}$
	on $V$ by
	\[
	\omega(\alpha_i,\alpha_j)=
	\begin{cases}
		2 & \text{if } i=j,\\
		-1 & \text{if there is a solid edge between } i \text{ and } j,\\
		2 & \text{if there is a double dotted edge between } i \text{ and } j,\\
		0 & \text{otherwise}.
	\end{cases}
	\]
	The matrix of the symmetric bilinear form $\omega(-,-)$ with respect to the basis $\Pi$ is called an {\em elliptic Cartan matrix}. This matrix is positive semi-definite with corank $2$.
	
	Analogous to the finite and affine cases, we define the simple reflections on $V$ by
	\[s_i(\alpha_j)=\alpha_j-\omega(\alpha_i,\alpha_j)\alpha_i,\] for $i,j\in I$.
	Let $W$ denote the subgroup of $\operatorname{GL}(V)$ generated by these simple reflections $s_i$ for $i\in I$, which is the Weyl group associated with $X_l^{(1,1)}$.
	Throughout this section, we adopt the following notations:
	\begin{itemize}
		\item The root lattice $\mathbf{Q}:=\mathbb{Z}\Pi=\sum_{i\in I}\mathbb{Z}\alpha_i$,
		\item The set of real roots $R^{re}:=W\Pi$,
		\item The set of imaginary roots $R^{im}:=\operatorname{rad}\omega(-,-)\cap (\mathbf{Q}\backslash\{0\})$.
	\end{itemize}
	For a sequence of elements $x_1, x_2,\dots x_n$ in a Lie algebra, we denote
	\[
	[x_1, x_2, \dots, x_n] := [[\dots, [[x_1, x_2], x_3], \dots, x_{n-1}],x_n].
	\]

	Associated with the elliptic Dynkin diagram of type $X_l^{(1,1)}$, the elliptic Lie algebra
	$\mathfrak{g}_{\mathrm{ell}}:=\mathfrak{g}_{\mathrm{ell}}(X_l^{(1,1)})$
	is presented by the following Chevalley generators and defining relations:
	\begin{itemize}
		\item[] Generators: $\{\alpha_i,e_{\pm i}\mid i\in I\}$;
		\item[] Relations:
		\begin{align}
			&[\alpha_i,\alpha_j]=0   \label{eq:rel-ell-1},  \\
			&[e_i,e_{-i}]=\alpha_i  \label{eq:rel-ell-2},    \\
			&[\alpha_i,e_j]= \omega(\alpha_i,\alpha_j)e_j  \label{eq:rel-ell-3},\\
			&(\operatorname{ad} e_i)^{\max\{1,\,1-\omega(\alpha_i,\alpha_j)\}}e_j=0,
			\label{eq:rel-ell-4}
		\end{align}
		where $\alpha_{-i}=-\alpha_i$, $i,j\in I$, and
		\begin{equation}
			\begin{aligned}[b]
				&[e_i,e_n,e_{1}]=0, \\
				&[e_{-i},e_{-n},e_{-1}]=0,
			\end{aligned}
			\quad \text{for} \quad
			\begin{tikzcd}[
				baseline={([yshift=0.3em]current bounding box.south)}, 
				row sep={2em, between origins}, 
				column sep=1.5em, 
				ampersand replacement=\&
				]
				\& 1 \\
				i \arrow[ur, no head] \arrow[r, no head] \& n, \arrow[u, dashed, shift left=0.5, no head] \arrow[u, dashed, shift right=0.5, no head]
			\end{tikzcd}
			\label{eq:rel-ell-5}
		\end{equation}
		
		\begin{equation}
			\begin{aligned}[b]
				&[e_i,e_n,e_j,e_{1}]=0, \\
				&[e_{-i},e_{-n},e_{-j},e_{-1}]=0,
			\end{aligned}
			\quad \text{for} \quad
			\begin{tikzcd}[
				baseline={([yshift=0.3em]current bounding box.south)}, 
				row sep={2em, between origins}, 
				column sep=1.5em, 
				ampersand replacement=\&
				]
				\& 1 \& \\
				i \arrow[ur, no head] \arrow[r, no head] \& n \arrow[u, dashed, shift left=0.5, no head] \arrow[u, dashed, shift right=0.5, no head] \arrow[r, no head] \& j. \arrow[ul, no head]
			\end{tikzcd}
			\label{eq:rel-ell-6}
		\end{equation}

	\end{itemize}
	The root lattice $\mathbf{Q}$ induces a natural $\mathbf{Q}$-grading on the elliptic Lie algebra:
	\[
	\mathfrak{g}=\bigoplus_{\alpha\in \mathbf{Q}}\mathfrak{g}_{\alpha}.  
	\]
	More precisely, this grading is defined by setting $\deg(e_{\pm i})=\alpha_{\pm i}$ and $\deg(\alpha_i)=0$ for $i\in I$.
	Saito--Yoshii \cite{SY00} showed that the dimensions of the corresponding root spaces are given by
	\[
	\dim_{\mathbb{C}}\mathfrak{g}_{\alpha}=
	\begin{cases}
		1, & \text{if } \alpha\in R^{\mathrm{re}},\\
		n-1, & \text{if } \alpha\in R^{\mathrm{im}},\\
		n, & \text{if } \alpha=0,\\
		0, & \text{otherwise}.
	\end{cases}  
	\]
	In particular, $\mathfrak{g}_0=\mathbb{C}\Pi$.

	\subsection{The quiver associated with $X_l^{(1,1)}$}\label{ss:tri-cat-ell-Lie}
	
	For each elliptic Dynkin diagram $X_l^{(1,1)}$ introduced in Subsection~\ref{ss:Elliptic Lie Algebra}, we construct an associated quiver $Q := Q(X_l^{(1,1)})$ as follows:
	\begin{itemize}
		\item The set of vertices is given by $Q_0:=I\cup \bar I=\{1,\ldots, n\}\cup\{\bar 1,\ldots,\bar n\}$.
		\item The set of arrows $Q_1$ is determined by the edges of the diagram $X_l^{(1,1)}$ for $i<j$ as follows:
		\begin{itemize}
			\item if there is a solid edge between $i$ and $j$, then we assign an arrow from $i$ to $j$, denoted by $\alpha_{ij}$, and an arrow from $\bar i$ to $\bar j$, denoted by $\alpha_{\bar i,\bar j}$;
			\item if there is a double dotted edge between $i$ and $j$, then we assign an arrow from $i$ to $\bar j$, denoted by $\alpha_{i\bar j}$, and an arrow from $\bar i$ to $j$, denoted by $\alpha_{\bar i j}$.
		\end{itemize}
	\end{itemize}
	By definition, the quiver $Q$ is acyclic. Let $\mathbb{F}$ be a finite field. For the path algebra $\mathbb{F}Q$, we define the admissible ideal
	\[
	\mathcal{I}:=\mathcal{I}(X_l^{(1,1)})=\Big\langle \sum_{j\in 
		J}\alpha_{1j}\alpha_{jn},
	\sum_{j\in J}\alpha_{\bar 1,\bar j}\alpha_{\bar j,\bar n} \Big\rangle,
	\]
	where $J:=\{j\in Q_0\mid \alpha_{1j},\alpha_{jn}\in Q_1\}$.
	
	Analogous to the construction in Subsection \ref{ss:valued-quiver-(C,D)}, we  define an involution $\theta$ on the quiver $Q$ by setting:
	\begin{itemize}
		\item $\theta(i)=\bar{i}$ and $\theta(\bar{i})=i$;
		\item $\theta(\alpha_{ij})=\alpha_{\bar{i},\bar{j}}$ and $\theta(\alpha_{i\bar{j}})=\alpha_{\bar{i}j}$.
	\end{itemize}
	It is evident that $\theta$ preserves the admissible ideal $\mathcal{I}$, and thus induces an involution on the quotient algebra $A:=\mathbb{F}Q/\mathcal{I}$. These finite-dimensional algebras associated with $X_l^{(1,1)}$ have global dimension $2$. For each vertex $i\in I\cup \bar I$, let $S_i$ denote the simple $A$-module at vertex $i$, and let $P(i)$ (resp.~$I(i)$) denote the indecomposable projective (resp.~injective) $A$-module at vertex $i$.
	
	\begin{example}\label{ep:D_4-quiver}
		We now illustrate this construction with the example of type $D_4^{(1,1)}$. Its elliptic Cartan matrix
		coincides exactly with the GIM presented in Example~\ref{ep:GIM-quiver}. The
		corresponding bound quiver $(Q(D_4^{(1,1)}),\mathcal{I}(D_4^{(1,1)}))$ is depicted in Figure~\ref{fig:quiver-D4}.
		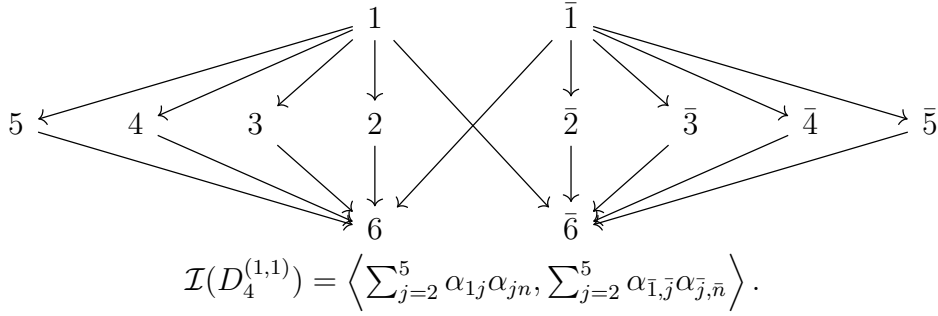
\begin{figure}
			\centering
			\begin{tikzcd}
				&&& 1 && {\bar 1} &&& \\
				5 & 4 & 3 & 2 && {\bar 2} & {\bar 3} & {\bar 4} & {\bar 5} \\
				&&& 6 && {\bar 6}
				\arrow[from=1-4, to=2-1]
				\arrow[from=1-4, to=2-2]
				\arrow[from=1-4, to=2-3]
				\arrow[from=1-4, to=2-4]
				\arrow[from=1-4, to=3-6]
				\arrow[from=1-6, to=2-6]
				\arrow[from=1-6, to=2-7]
				\arrow[from=1-6, to=2-8]
				\arrow[from=1-6, to=2-9]
				\arrow[from=1-6, to=3-4]
				\arrow[from=2-1, to=3-4]
				\arrow[from=2-2, to=3-4]
				\arrow[from=2-3, to=3-4]
				\arrow[from=2-4, to=3-4]
				\arrow[from=2-6, to=3-6]
				\arrow[from=2-7, to=3-6]
				\arrow[from=2-8, to=3-6]
				\arrow[from=2-9, to=3-6]
			\end{tikzcd}
			$\mathcal{I}(D_4^{(1,1)})
			=
			\left\langle
			\sum_{j=2}^5\alpha_{1j}\alpha_{j n},
			\sum_{j=2}^5\alpha_{\bar 1,\bar j}\alpha_{\bar j,\bar n}
			\right\rangle .$
			\caption{The bound quiver of $D_4^{(1,1)}$.}
			\label{fig:quiver-D4}
		\end{figure}
		
		
		In the case of the GIM algebra considered
		in Example \ref{ep:GIM-quiver}, each arrow from $1$ to $\bar 6$ contributes $1$ to the Euler form
		$(h_{S_1},h_{S_6})_{\mathcal{D}/G}$ (see the computation in
		Lemma~\ref{lem:Euler-form-D/G}). In the present case, after quotienting by
		the ideal $\mathcal{I}(D_4^{(1,1)})$, $\dim_{\mathbb{F}}\Ext^2_{A}(S_1,S_6)$ also contributes $1$ to the Euler form $(h_{S_1},h_{S_6})$ (cf. Lemma~\ref{lem:Euler-form-M} below). Consequently, to ensure that
		$(h_{S_1},h_{S_6})=2$, the quiver $Q\bigl(D_4^{(1,1)}\bigr)$ features only a single arrow from $1$ to $\bar{6}$. This stands in contrast to the quiver in Figure~\ref{fig:quiver-GIM}, which requires two parallel arrows from $1$ to $\bar{6}$ to achieve the same inner product.

	\end{example}

	\subsection{2-periodic triangulated categories associated with $X_{l}^{(1,1)}$}
	Let
	\[A:=A(X_l^{(1,1)})
	=
	kQ(X_l^{(1,1)})/\mathcal{I}(X_l^{(1,1)})\]
	be the finite-dimensional $\mathbb{F}$-algebra, equipped with the
	involution $\theta$ defined above. Following the framework established in Section~\ref{s:2-periodic-M}, one can construct the associated $2$-periodic triangulated category $\mathcal{M}_{\theta}$. Throughout the remainder of this section, we suppress $\theta$ from the notation and write simply $\mathcal{M}$. We shall also adopt the standard notations introduced in Section~\ref{s:2-periodic-M} for triangulated categories and Ringel–Hall Lie algebras.

	By Lemma~\ref{lem:basis-go-M}, the elements
	$h_{S_1},\ldots,h_{S_n}$
	form a $\mathbb{Z}$-basis of the Grothendieck group
	$\go(\mathcal{M})$ of $\mathcal{M}$. Let
	$(-,-)_{\mathcal{M}}$ denote the symmetric Euler form of
	$\mathcal{M}$. We then have the following result.

	\begin{lemma}\label{lem:Euler-form-M}
		The matrix of $(-,-)_{\mathcal{M}}$, with respect to the basis $h_{{S}_{1}},\ldots,h_{{S}_{n}}$ of $\go(\mathcal{M})$, is precisely the elliptic Cartan matrix associated with $X_l^{(1,1)}$.
	\end{lemma}
	\begin{proof}
		We restrict our attention to the case $X_l^{(1,1)} = D_4^{(1,1)}$, as the remaining three types can be handled in an identical manner.  The following projective resolutions of the simple $A$-modules can be obtained from the definition of $A$ (see Figure \ref{fig:quiver-D4}):
		\begin{align}
			0\rightarrow P(6) \rightarrow & P(2)\oplus P(3)\oplus P(4)\oplus P(5)\oplus P(\bar 6) \rightarrow P(1) \rightarrow S_1\to0,\label{eq:pro-res-simple 1}\\
			&0\rightarrow  P(6) \rightarrow S_6\to0,\label{eq:pro-res-simple 6}\\
			&0\rightarrow P(6) \rightarrow P(i) \rightarrow S_i\rightarrow 0, \quad (i=2,3,4,5)\label{eq:pro-res-simple i}.
		\end{align}
		Observe that $\gldim A=2$. It follows that
		\begin{eqnarray*}
			\Hom_{\mathcal{M}}(S_i,S_j)
			&=& \Hom_{\mathcal{D}/G}(S_i,S_j)=\bigoplus_{p\in \mathbb{Z}}\Hom_{\mathcal{D}}(S_i,G^pS_j)\\
			&=&\Hom_{\mathcal{D}}(S_i,S_j)\oplus \Hom_{\mathcal{D}}(S_i,\Sigma S_{\bar{j}})\oplus \Hom_{\mathcal{D}}(S_i,\Sigma^2 S_j),
		\end{eqnarray*}
		and completely analogously,
		\[\Hom_{\mathcal{M}}(S_i,\Sigma S_j)
		=\Hom_{\mathcal{D}}(S_i,S_{\bar{j}})\oplus \Hom_{\mathcal{D}}(S_i,\Sigma S_j)\oplus \Hom_{\mathcal{D}}(S_i,\Sigma^2S_{\bar j}).\]
		Consequently, a direct computation yields
		\begin{eqnarray*}
			(h_{S_i},h_{S_j})_{\mathcal{M}}
			&=& \dim \Hom_{\mathcal{D}/G}(S_i,S_j)-\dim \Hom_{\mathcal{D}/G}(S_i,\Sigma S_j)\\
			&& +\dim \Hom_{\mathcal{D}/G}(S_j,S_i)-\dim \Hom_{\mathcal{D}/G}(S_j,\Sigma S_i)\\
			&=& \begin{cases}
				2\dim \Hom_{\mathcal{D}}(S_i,S_i)=2 & \text{ if }i=j, \\
				-\dim \Hom_{\mathcal{D}}(S_i,\Sigma S_j)=-\#\{i\rightarrow j\} & \text{ if }i<j,(i,j)\ne (1,n),\\
				\dim \Hom_{\mathcal{D}}(S_1,\Sigma S_{\bar n})+\dim\Hom_{\mathcal{D}}(S_1,\Sigma^2 S_n)=2 & \text{ if }(i,j)=(1,n),
			\end{cases}
		\end{eqnarray*}
		where $\#\{i\rightarrow j\}$ denotes the number of arrows from $i$ to $j$ in the quiver $Q(D_4^{(1,1)})$. The remainder of the proof follows immediately.
	\end{proof}
	
	We obtain the following corollary.
	\begin{corollary}\label{cor:Q=K0(M)}
		There is a group isomorphism
		\begin{align*}
			\eta:\go(\mathcal{M}) &\longrightarrow \mathbf{Q},\\
			h_{S_i} &\longmapsto \alpha_i,\qquad i\in I.
		\end{align*}
		Moreover, under $\eta$, the bilinear form $(-,-)_{\mathcal{M}}$ coincides with $\omega(-,-)$.
	\end{corollary}

	\subsection{The Ringel--Hall Lie algebra $\mathfrak{g}(\mathcal{M})_{(q-1)}$ }
	Let $|\mathbb{F}|=q$. Applying the construction in
	Subsection~\ref{ss-R-H-Lie}, we obtain the Lie algebra
	$\mathfrak{g}(\mathcal{M})_{(q-1)}$ over
	$\mathbb{Z}/(q-1)\mathbb{Z}$. In what follows, we establish several lemmas that will be crucial for proving Theorem~\ref{thm:relization-ell-Lie} later in this section.
	
	Recall that by Lemma ~\ref{lem:theta-in-M}, we have a functor $\theta:\mathcal{M}\rightarrow\mathcal{M}$.
	By the definition of $\mathfrak{g}(\mathcal{M})_{(q - 1)}$, we immediately obtain the following result. 
	
	\begin{lemma}
		The functor $\theta$ on $\mathcal{M}$
		induces an involution
		$\theta:
		\mathfrak{g}(\mathcal{M})_{(q-1)}
		\rightarrow
		\mathfrak{g}(\mathcal{M})_{(q-1)}.$
		In particular,
		$\theta(h_X)=h_{\theta(X)}
		$ and $\theta(u_X)=u_{\theta(X)}$.
	\end{lemma}

	For each $i\in I$, set $u_{S_{-i}}:=u_{\Sigma S_i}=u_{S_{\bar i}}$. Then the elements $u_{S_{\pm i}}$, for $i\in I$, satisfy the Serre relation \eqref{eq:rel-ell-4} of the elliptic Lie algebra.

	\begin{lemma}\label{lem:ell-serre-relation}
		The following relations hold in the Lie algebra $\mathfrak{g}(\mathcal{M})_{(q-1)}$:
		\[
		(\operatorname{ad} u_{S_i})^{\max\{1,\,1-\omega(\alpha_i,\alpha_j)\}}u_{S_j}=0,\qquad i\ne j\in \pm I. 
		\]
	\end{lemma}
	\begin{proof}
		As before, we restrict our verification to the case $X_l^{(1,1)}=D_4^{(1,1)}$; the cases $X_l^{(1,1)}=E_6^{(1,1)},E_7^{(1,1)},E_8^{(1,1)}$ can be handled in a completely analogous manner.
		
		When $i=\pm1$ and $j=\pm6$, it suffices to prove
		\begin{align}
			& [u_{S_1},u_{S_6}]=0=[u_{S_{\bar 1}},u_{S_{\bar 6}}], \label{eq:serre-ell-Lie-1}\\
			& (\operatorname{ad} u_{S_1})^3 u_{S_{\bar 6}}=0=(\operatorname{ad} u_{S_{\bar 1}})^3 u_{S_6}. \label{eq:serre-ell-Lie-2}
		\end{align}
		
		From the projective resolutions of the simple modules $S_i$, given in \eqref{eq:pro-res-simple 1}--\eqref{eq:pro-res-simple i}, one readily obtains
		\begin{align*}
			\Hom_{\mathcal M}(S_6,\Sigma S_1)
			&=
			\Hom_{\mathcal D}(S_6,S_{\bar 1})
			\oplus
			\Hom_{\mathcal D}(S_6,\Sigma S_1)\oplus \Hom_{\mathcal D}(S_6,\Sigma^2 S_{\bar{1}})
			=0, \\
			\Hom_{\mathcal M}(S_1,\Sigma S_6)
			&=
			\Hom_{\mathcal D}(S_1,S_{\bar 6})
			\oplus
			\Hom_{\mathcal D}(S_1,\Sigma S_6)\oplus \Hom_{\mathcal D}(S_1,\Sigma^2 S_{\bar 6})
			=0.
		\end{align*}
		Therefore, for any indecomposable object $L\in\mathcal M$, we have
		$\gamma_{S_1,S_6}^L=0$. Moreover, by Lemma~\ref{lem:basis-go-M},
		we have $S_1\not\simeq \Sigma S_6$. Hence
		$[u_{S_1},u_{S_6}]=0$. Applying the involution $\theta$ on
		$\mathfrak{g}(\mathcal{M})_{(q-1)}$, we also obtain
		$[u_{S_{\bar 1}},u_{S_{\bar 6}}]=0$. Therefore,
		\eqref{eq:serre-ell-Lie-1} follows.
		
		Now we verify the relation $(\operatorname{ad} u_{S_1})^3 u_{S_{\bar 6}}=0$. As in the proof of Proposition~\ref{pro:Serre-relation}, we recursively define collections of objects $\mathcal{X}_t\subseteq \mathcal{M}$ for $t=0,1,2$ as follows:
		\begin{itemize}
			\item $\mathcal{X}_0=\{S_{\bar 6}\}$;
			\item Suppose that $\mathcal{X}_{t-1}$ has been defined. Then
			\[
			\mathcal{X}_t=\{X\in \mathcal{M}\mid \exists\text{ a triangle } X_{t-1}\rightarrow X\rightarrow S_1\rightarrow \Sigma X_{t-1} \text{ in } \mathcal{M}, \text{ with } X_{t-1}\in \mathcal{X}_{t-1}\}.  
			\]
		\end{itemize}
		
		In this situation, assertion $(a)$ in Proposition \ref{pro:Serre-relation} still holds:
		\begin{enumerate}
			\item[(a)] $\Hom_{\mathcal{M}}(X_t,\Sigma S_1)=0$ for any $X_t\in \mathcal{X}_t$ and $t=0,1,2$. Consequently, every triangle of the form
			$S_1\rightarrow L\rightarrow X_t\rightarrow \Sigma S_1$
			in $\mathcal{D}/G$ is split, where $X_t\in \mathcal{X}_t$.
		\end{enumerate}
		Moreover, its proof is exactly the same as that of Proposition~\ref{pro:Serre-relation}.
		
		However, assertion $(b)$ no longer holds in the present situation and must be modified. Indeed, we have the following statement:
		\begin{enumerate}
			\item[(b$'$)] For two indecomposable objects $X_1\in \mathcal{X}_1$ and $X_2\in \mathcal{X}_2$, suppose that there are triangles
			\begin{align}
				\Sigma^{-1}S_1\overset{f}{\longrightarrow} S_{\bar 6}\longrightarrow X_1\longrightarrow S_1, \label{eq:ell-serre-relation-extention-1}\\
				\Sigma^{-1}S_1\overset{g}{\longrightarrow} X_1\longrightarrow X_2\longrightarrow S_1. \label{eq:ell-serre-relation-extention-2}
			\end{align}
			Then $\Hom_{\mathcal M}(\Sigma^{-1}S_1,X_2)=0$.
		\end{enumerate}
		We now prove claim $(b')$. Note that $S_1=I(1)$, and that $S_{\bar 6}$ admits the following injective resolution:
		\[
		0\longrightarrow S_{\bar 6}\longrightarrow I(\bar 6)\longrightarrow I(1)\oplus I(\bar 2)\oplus I(\bar 3)\oplus I(\bar 4)\oplus I(\bar 5)\longrightarrow I(\bar 1)\longrightarrow0.
		\]
		For
		$f\in \Hom_{\mathcal M}(\Sigma^{-1}S_1,S_{\bar 6}),$
		as explained in \eqref{eq:pi-theta-X-2-periodic-complex}, the morphism
		$\pi_\theta f:\pi_\theta(\Sigma^{-1}S_1)\rightarrow \pi_\theta S_{\bar 6}$ can be represented by the following chain map of
		$\theta$-complexes of
		injective modules:
		\[
		\xymatrix{ 
			\cdots &
			{I(1)}\ar[d]_{(k_{11},k_{21},0)^t}\ar[r]^0
			& {I(\bar 1)}\ar[d]_{(k_{11},k_{21},0)^t}\ar[r]^0
			& {I(1)}\ar[d]^{(k_{11},k_{21},0)^t}
			&\cdots \\
			\cdots & 
			{I(1)^2\oplus \theta I}
			\ar[r]^{\theta\psi_{0}}
			& I(\bar 1)^2\oplus I
			\ar[r]^{\psi_0}
			& {I(1)^2\oplus \theta I}
			&\cdots .
		}
		\]
		Hereafter, each $k_{ij} \in \mathbb{F}$ is understood to act as an endomorphism on its corresponding object.
		Hence, by a direct computation, we obtain
		$\pi_\theta X_1
		=
		\cone\bigl(
		\pi_\theta(\Sigma^{-1}S_1)
		\xrightarrow{\,\pi_\theta f\,}
		\pi_\theta S_{\bar 6}
		\bigr)$
		is of the form
		\[
		\pi_{\theta}X_1:
		\cdots \longrightarrow
		I(\bar 1)\oplus I(1)^2\oplus \theta I
		\xrightarrow{\theta \psi_1}
		I(1)\oplus I(\bar 1)^2\oplus  I
		\xrightarrow{\psi_1}
		I(\bar 1)\oplus I(1)^2\oplus \theta I
		\longrightarrow \cdots,
		\]
		where
		$I=I(2)\oplus I(3)\oplus I(4)\oplus I(5)\oplus I(\bar 6)$,
		and
		\[
		\psi_1 =\begin{pmatrix}
			0 & 0 & 0 & 0\\
			k_{11} & * & * & * \\
			k_{21} & * & * & * \\
			0 & * & * & *
		\end{pmatrix}.
		\]
		Since
		$X_1\in\mathcal M$ is indecomposable, we necessarily have $f\neq 0$.
		Therefore $\pi_\theta f\neq 0$ by Lemma~\ref{lem:pi-theta-fully-fathful}, and hence
		$(k_{11},k_{21})^t\neq 0$.
		
		Applying $\pi_\theta$ to the triangle \eqref{eq:ell-serre-relation-extention-2}, we obtain the following chain map $\pi_\theta g: \pi_\theta (\Sigma^{-1}S_1)\rightarrow \pi_\theta X_1$ between $\theta$-complexes of injective modules:
		\[
		\xymatrix{ 
			\cdots &
			{I(1)}\ar[d]_{(0,k_{12},k_{22},0)^t}\ar[r]^0 &  {I(\bar 1)} \ar[d]_{(0,k_{12},k_{22},0)^t}\ar[r]^0& {I(1)}\ar[d]^{(0,k_{12},k_{22},0)^t} & \cdots \\
			\cdots & 
			{I(\bar 1)\oplus I(1)^2\oplus \theta I} 
			\ar[r]^{
				\theta \psi_{1}
			} 
			& I(1) \oplus  I(\bar 1)^2\oplus  I\ar[r]^{ \psi_1} &{I(\bar 1) \oplus I(1)^2\oplus \theta I} & \cdots
		}
		\]
		It follows immediately that
		$\pi_\theta X_2
		=\cone\bigl(\pi_\theta\Sigma^{-1}S_1\xrightarrow{\, \pi_\theta g \,} \pi_\theta X_1\bigr)$
		is of the form 
		\[
		\pi_{\theta}X_2:
		\cdots \longrightarrow
		I(\bar 1)^2\oplus I(1)^2\oplus \theta I
		\xrightarrow{\theta \psi_2}
		I(1)^2\oplus I(\bar 1)^2\oplus  I
		\xrightarrow{\psi_2}
		I(\bar 1)^2\oplus I(1)^2\oplus \theta I
		\longrightarrow \cdots,
		\]
		where 
		\[
		\psi_2 =\begin{pmatrix}
			0 & 0 & 0 & 0 & 0\\
			0 & 0 & 0 & 0 & 0\\
			k_{12} & k_{11} & * & * & * \\
			k_{22} & k_{21} & * & * & * \\
			0 & 0 & * & * & *
		\end{pmatrix}.
		\]
		We claim that $(k_{12},k_{22})^t$ and $(k_{11},k_{21})^t$ are linearly independent. Otherwise,
		assume that
		$(k_{12},k_{22})^t=a(k_{11},k_{21})^t$ for some $a\neq \mathbb{F}$. It follows that
		the morphism
		\[
		(a,0,0,0): I(1)\longrightarrow I(1)\oplus I(\bar 1)^2\oplus I
		\]
		induces a homotopy $\pi_\theta g\sim 0$. Since $X_2$ is indecomposable, we must have $\pi_\theta g\neq 0$, a contradiction. Hence, the vectors $(k_{12},k_{22})^t$ and $(k_{11},k_{21})^t$ are linearly independent.

		Let
		$h\in \Hom_{\mathcal M}(\Sigma^{-1}S_1,X_2)$.
		Then $\pi_\theta h: \pi_\theta(\Sigma^{-1}S_1)\rightarrow \pi_\theta X_2$ can be represented by the following chain map  of
		$2$-periodic complexes complexes of injective modules:
		\[
		\xymatrix{ 
			\cdots &
			{I(1)}\ar[d]_{(0,0,k_{1},k_{2},0)^t}\ar[r]^0 &  {I(\bar 1)} \ar[d]_{(0,0,k_{1},k_{2},0)^t}\ar[r]^0& {I(1)}\ar[d]^{(0,0,k_{1},k_{2},0)^t} &\cdots \\
			\cdots  &
			{I(\bar 1)^2\oplus I(1)^2\oplus \theta I} 
			\ar[r]^{
				\theta \psi_{2}
			} 
			& I(1)^2 \oplus  I(\bar 1)^2\oplus  I\ar[r]^{ \psi_2} &{I(\bar 1)^2 \oplus I(1)^2\oplus \theta I} & \cdots.
		}
		\]
		Since the vectors $(k_{12},k_{22})^t$ and $(k_{11},k_{21})^t$ are linearly independent, there exist $b,c\in\mathbb{F}$ such that
		$(k_1,k_2)^t=b(k_{12},k_{22})^t+c(k_{11},k_{21})^t.$
		It follows that the morphism
		\[
		(b,c,0,0,0)^t: I(1)\longrightarrow I(1)^2\oplus I(\bar 1)^2\oplus I
		\]
		induces a homotopy $\pi_\theta h\sim 0$. Consequently, $h=0$. This complete the proof of $(b')$.
		
		As in the discussion in Proposition~\ref{pro:Serre-relation}, we see that no term of the form $h_{S_1}$ occurs in the computation of
		$(\operatorname{ad}u_{S_1})^{3}u_{S_{\bar 6}}$. Moreover, by Claims $(a)$ and $(b')$, we obtain
		$(\operatorname{ad} u_{S_1})^3 u_{S_{\bar 6}}=0.$
		Then
		$(\operatorname{ad} u_{S_{\bar 1}})^3 u_{S_6}=0$
		follows by applying the involution $\theta$ of $\mathfrak{g}(\mathcal{M})_{(q-1)}$ to the relation
		$(\operatorname{ad} u_{S_1})^3 u_{S_{\bar 6}}=0.$
		This completes the proof of \eqref{eq:serre-ell-Lie-2}.
		
		The case where $i = \pm 6$ and $j = \pm 1$ is completely dual; in this situation, the verification is carried out analogously by utilizing the projective resolutions instead. For all other pairs of indices $i$ and $j$, the arguments established in the proof of Proposition~\ref{pro:Serre-relation} carry over without change. This completes the proof of the lemma.
	\end{proof}

	The elements $u_{S_{\pm i}}$, for $i\in I$, also satisfy the Serre relations \eqref{eq:rel-ell-5} and \eqref{eq:rel-ell-6} for the elliptic Lie algebra.
	
	\begin{lemma}\label{lem:ell-relation}
		The following relations hold in the Lie algebra $\mathfrak{g}(\mathcal{M})_{(q-1)}$:
		\begin{equation*}
			\begin{aligned}[b]
				&[u_{S_i},u_{S_n},u_{S_1}]=0,\\
				&[u_{S_{\bar i}},u_{S_{\bar n}},u_{S_{\bar 1}}] = 0,
			\end{aligned}
			\quad \text{for} \quad
			\begin{tikzcd}[
				baseline={([yshift=0.3em]current bounding box.south)}, 
				row sep={2em, between origins}, 
				column sep=1.5em, 
				ampersand replacement=\&
				]
				\& 1 \\
				i \arrow[ur, no head] \arrow[r, no head] \& n, \arrow[u, dashed, shift left=0.5, no head] \arrow[u, dashed, shift right=0.5, no head]
			\end{tikzcd}
		\end{equation*}
		and
		\begin{equation*}
			\begin{aligned}[b]
				&[u_{S_i},u_{S_n},u_{S_j},u_{S_1}]=0,\\
				&[u_{S_{\bar i}},u_{S_{\bar n}},u_{S_{\bar j}},u_{S_{\bar 1}}]=0,
			\end{aligned}
			\quad \text{for} \quad
			\begin{tikzcd}[
				baseline={([yshift=0.3em]current bounding box.south)}, 
				row sep={2em, between origins}, 
				column sep=1.5em, 
				ampersand replacement=\&
				]
				\& 1 \& \\
				i \arrow[ur, no head] \arrow[r, no head] \& n \arrow[u, dashed, shift left=0.5, no head] \arrow[u, dashed, shift right=0.5, no head] \arrow[r, no head] \& j. \arrow[ul, no head]
			\end{tikzcd}
		\end{equation*}
		
	\end{lemma}
	
	\begin{proof}
		As before, we restrict our verification to the case $X_l^{(1,1)}=D_4^{(1,1)}$ and verify the relation $[u_{S_i},u_{S_6},u_{S_1}]=0$, the other cases can be established similarly.
		
		We first compute $[u_{S_i},u_{S_6}]$. Note that $1<i<6$. By the projective resolutions \eqref{eq:pro-res-simple 1}--\eqref{eq:pro-res-simple i}, one readily obtains
		\[
		\dim\Hom_{\mathcal{M}}(S_6,\Sigma S_i)=0,
		\qquad
		\dim\Hom_{\mathcal{M}}(S_i,\Sigma S_6)=1.
		\]
		Note that $P(i)$ is the unique indecomposable $A$-module arising as an extension of the simple modules $S_i$ and $S_6$. Hence
		$F_{S_iS_6}^{P(i)}=1,
		F_{S_6S_i}^{P(i)}=0,$
		and therefore
		$[u_{S_i},u_{S_6}]=-u_{P(i)}.$ Next, we compute $[u_{P(i)},u_{S_1}]$. It is easy to see that
		\[
		\Hom_{\mathcal{M}}(S_1,\Sigma P(i))=\Hom_{\mathcal{M}}(P(i),\Sigma S_1)=0.
		\]
		Therefore,
		$[u_{P(i)},u_{S_1}]=0.$
		This proves the lemma.
	\end{proof}

	\subsection{Elliptic algebras via Ringel--Hall Lie algebras}
	Let $\mathbb{F}$ be a finite field with $|\mathbb{F}|=q$, let $A$ denote the $\mathbb{F}$-algebra defined by the bound quiver in Section \ref{ss:tri-cat-ell-Lie}, and let $\mathcal{M}$ be the associated $2$-periodic triangulated category. Then, as in Section \ref{ss-R-H-Lie}, one may define the Ringel--Hall Lie algebra $\mathfrak{g}(\mathcal{M})_{(q-1)}$. 
	
	In this setting, for every finite field extension $\mathbb{F}\subseteq \mathbb{K}$, and for every simple $A$-module $S_i$, we have
	$\operatorname{End}_A(S_i)/\operatorname{rad}\operatorname{End}_A(S_i)=\mathbb{F}.$
	Hence $\mathbb{K}$ is conservative for all simple $A$-modules. We therefore set
	\[
	\Omega=\{\mathbb{K}\mid \mathbb{F}\subseteq \mathbb{K}\subseteq \ov{\mathbb{F}} \text{ is a finite field extension}\}.    
	\]
	Let $\mathcal{M}^{\mathbb{K}}$ denote the triangulated hull of the orbit category $\mathcal{D}^b(\mod A^{\mathbb{K}})/G$, and consider the direct product of the following Lie algebras:
	\[
	\prod_{\mathbb{K}\in \Omega}\mathfrak{g}(\mathcal{M}^{\mathbb{K}})_{(|\mathbb{K}|-1)}.    
	\]
	Denote by $\mathscr{LC}(\mathcal{M})$ the $\mathbb{Z}$-subalgebra generated by
	\begin{align*}
		{\bf u}_{S_i}:=(u_{{S_i}^\mathbb{K}})_{\mathbb{K}\in\Omega},\quad
		{\bf u}_{ S_{\bar i}}:=
		(u_{ {S_{\bar i}}^\mathbb{K}})_{\mathbb{K}\in\Omega},\quad
		{\bf h}_{S_i}:=(h_{{S_i}^\mathbb{K}})_{\mathbb{K}\in\Omega},
		\qquad i\in I.
	\end{align*}
	
	By definition, $\mathscr{LC}(\mathcal{M})$ carries a natural $\go(\mathcal{M})$-grading. Hence, by Corollary \ref{cor:Q=K0(M)}, this induces a $\mathbf{Q}$-grading on $\mathscr{LC}(\mathcal{M})$, under which
	\[
	\deg({\bf u}_{S_i})=\alpha_i,\qquad
	\deg({\bf u}_{ S_{\bar i}})=-\alpha_i,\qquad
	\deg({\bf h}_{S_i})=0,\qquad i\in I.
	\]
	
	The following is the main result of this section.
	
	\begin{theorem}\label{thm:relization-ell-Lie}
		Let $\mathfrak{g}_{\mathrm{ell}}$ be an elliptic Lie algebra of type $D_4^{(1,1)}$, $E_6^{(1,1)}$,  and let $\mathscr{LC}(\mathcal{M})$ be the associated integral Ringel--Hall Lie algebra. Then there exists a well-defined surjective Lie algebra homomorphism
		\begin{align*}
			\Theta:\,&\mathfrak{g}_{\mathrm{ell}} \longrightarrow \mathscr{LC}(\mathcal{M})\otimes_{\mathbb{Z}}\mathbb{C},\\
			&e_i \longmapsto {\bf u}_{S_i},\\
			&e_{-i} \longmapsto -{\bf u}_{ S_{\bar i}},\\
			&\alpha_i \longmapsto {\bf h}_{S_i},
		\end{align*}
		for all $i\in I$.
		Moreover, the homomorphism $\Theta$ preserves the $\mathbf{Q}$-grading, and its restriction to each root space $(\mathfrak{g}_{\mathrm{ell}})_\alpha$ is injective for all $\alpha\in R^{\mathrm{re}}\cup\{0\}$.
	\end{theorem}
	
	\begin{proof}
		To show that $\Theta$ is well-defined, it suffices to verify the defining relations~\eqref{eq:rel-ell-1}--\eqref{eq:rel-ell-6} one by one. Relation~\eqref{eq:rel-ell-1} follows immediately from the definition. For relation~\eqref{eq:rel-ell-2}, it is enough to verify that
		$[\Theta(e_i),\Theta(e_{-i})]=\Theta(\alpha_i),$
		whose proof is completely analogous to that of Theorem~\ref{thm:GIM-via-RH-Lie}. The remaining relations~\eqref{eq:rel-ell-3}--\eqref{eq:rel-ell-6} follow respectively from Corollary~\ref{cor:Q=K0(M)}, Lemmas~\ref{lem:ell-serre-relation}--\ref{lem:ell-relation}. Therefore, by the definition of $\mathscr{LC}(\mathcal{M})\otimes_{\mathbb{Z}}\mathbb{C}$, the surjectivity of $\Theta$ is clear.
		
		By Corollary~\ref{cor:Q=K0(M)}, the restriction $\Theta|_{(\mathfrak{g}_{\mathrm{ell}})_0}$ is injective. For each simple root $\alpha_i$, $i\in I$, it is clear that $\Theta|_{(\mathfrak{g}_{\mathrm{ell}})_{\alpha_i}}$ is injective. Note that $\operatorname{ad} {\bf u}_{S_i}$ and $\operatorname{ad} {\bf u}_{ S_{\bar i}}$ are locally nilpotent. Hence
		$\operatorname{exp}(\operatorname{ad} {\bf u}_{ S_i})
		\operatorname{exp}(\operatorname{ad} {\bf u}_{S_{\bar i}})
		\operatorname{exp}(\operatorname{ad} {\bf u}_{ S_{ i}})$
		induces an isomorphism
		$(\mathscr{LC}(\mathcal{M})\otimes_{\mathbb{Z}}\mathbb{C})_\alpha
		\cong
		(\mathscr{LC}(\mathcal{M})\otimes_{\mathbb{Z}}\mathbb{C})_{s_i\alpha}.$
		It follows that for every $\alpha\in R^{\mathrm{re}}$, we have
		$\dim (\mathscr{LC}(\mathcal{M})\otimes_{\mathbb{Z}}\mathbb{C})_\alpha=1,$
		and hence $\Theta|_{(\mathfrak{g}_{\mathrm{ell}})_\alpha}$ is injective. This completes the proof.
	\end{proof}
	
	\begin{example}
		We discuss the injectivity of the homomorphism $\Theta$ on imaginary
		root spaces, taking the type $D_4^{(1,1)}$ as an example. Let
		$a\in R^{\mathrm{im}}$ be the imaginary root such that
		\[
		(\mathfrak{g}_{\mathrm{ell}})_a
		=
		\mathbb{C}[e_1,e_{-6}]
		\oplus
		\bigoplus_{i=2}^5
		\mathbb{C}\bigl[[e_1,e_i],[e_{-6},e_{-i}]\bigr].
		\]
		We compute, in the Ringel--Hall algebra
		$\mathfrak{g}(\mathcal{M})_{(q-1)}$, the elements
		$\bigl[[u_{S_1},u_{S_{i}}],
		[u_{S_{\bar{6}}},u_{S_{\bar i}}]\bigr]
		$ for $i=2,3,4,5$.
		It is easy to see that
		\[
		[u_{S_1},u_{S_i}]=-u_{I(i)},
		\qquad
		[u_{S_{\bar 6}},u_{S_{\bar i}}]=u_{P(\bar i)}.
		\]
		We next consider the commutator $[u_{I(i)},u_{P(\bar i)}]$. Since
		\[
		\dim\Hom_{\mathcal{M}}(P(\bar i),\Sigma I(i))=\dim \Hom_{\mathcal{M}}(I(i),\Sigma P(\bar i))=1,
		\]
		it follows that
		\[
		[u_{I(i)},u_{P(\bar{i})}]=-u_{M(i)}+u_{M'(i)},
		\]
		where $M(i)$ is the unique
		indecomposable $A(D_4^{(1,1)})$-module obtained as an extension of  $I(i)$ by $P(\bar i)$, and $M'(i)$ is the indecomposable complex $\cdots\rightarrow 0\rightarrow P(\bar i)\rightarrow I(\bar i)\rightarrow 0\rightarrow \cdots$ whose degree-zero term is $P(\bar i)$. Moreover,
		\[
		\widehat{H^0(\pi_\rho M(i))}
		=
		\hat S_1+\hat S_{\bar 6}+\hat S_i+\hat S_{\bar i}.  
		\]
		Now, let us consider the commutator $[u_{S_1}, u_{S_{\bar{6}}}]$. Let $N$ be the unique indecomposable module of dimension $2$ with socle $S_{\bar{6}}$ and top $S_1$. Since $S_{\bar{6}}$ is projective, a straightforward calculation shows that the coefficient of $u_{N}$ in the expansion of $[u_{S_1}, u_{S_{\bar{6}}}]$ is exactly $-1$. Consequently, the commutator $[u_{S_1}, u_{S_{\bar{6}}}]$ is nonzero.
		For any triangle in $\mathcal{M}$ of the form
		\begin{align*}
			S_1 \longrightarrow X \longrightarrow
			S_{\bar 6}\longrightarrow \Sigma S_1, \\
			S_{\bar 6}\longrightarrow X \longrightarrow
			S_1\longrightarrow \Sigma  S_{\bar 6},
		\end{align*}
		applying $\pi_\rho$ gives, respectively, exact sequences
		\begin{align*}
			H^0(\pi_\rho S_1)\longrightarrow H^0(\pi_\rho X)\longrightarrow
			H^0(\pi_\rho  S_{\bar 6}), \\
			H^0(\pi_\rho  S_{\bar 6})\longrightarrow H^0(\pi_\rho X)\longrightarrow
			H^0(\pi_\rho S_1),
		\end{align*}
		Thus, for any term $u_{X}$ occurring in
		$[u_{S_1},u_{S_{\bar 6}}]$, we have
		\[
		\widehat{H^0(\pi_\rho X)}
		\in
		\mathbb{Z}\hat S_1
		\oplus
		\mathbb{Z}\hat S_{\bar 6}.
		\]
		Consequently,
		$[u_{S_1},u_{S_{\bar 6}}]$ and $[[u_{S_1},u_{S_{\bar 6}}], [u_{S_i},u_{S_{\bar i}}]]$ ($i=2,3,4,5$) are linearly independent. It follows that the restriction of $\Theta$
		to the imaginary root space $(\mathfrak{g}_{\mathrm{ell}})_a$ is
		injective.
		
		However, for general imaginary root spaces, for example
		$(\mathfrak{g}_{\mathrm{ell}})_{na}$ with $n\in\mathbb{Z}$, the
		corresponding elements in the Ringel--Hall algebra are difficult to
		compute. Hence, we do not have a suitable method to prove the injectivity of $\Theta$.
	\end{example}

	\begin{remark}
		For other types of simply-laced elliptic Lie algebras, for example the
		type $D_5^{(1,1)}$, one encounters defining relations of the following
		form:
		\begin{equation}
			\begin{aligned}
				&[e_{-3},e_1,e_4]=e_2,\\
				\\
				&[e_3,e_{-1},e_{-4}]=e_{-2},
			\end{aligned}
			\qquad\text{for}\qquad
			\begin{tikzcd}[
				row sep=1.0cm,
				column sep=1.0cm,
				]
				1 & 2  \\
				3 & 4.
				\arrow[no head, from=1-1, to=1-2]
				\arrow[no head, from=1-1, to=2-2]
				\arrow[no head, from=1-2, to=2-1]
				\arrow[no head, from=2-1, to=2-2]
				\arrow[shift left, dashed, no head, from=2-2, to=1-2]
				\arrow[shift right, dashed, no head, from=2-2, to=1-2]
				\arrow[shift left, dashed, no head, from=1-1, to=2-1]
				\arrow[shift right, dashed, no head, from=1-1, to=2-1]
			\end{tikzcd}
		\end{equation}
		If we still construct the corresponding $2$-periodic triangulated
		category $\mathcal M$ in a way analogous to the one above, then the
		relations displayed above are not homogeneous in the Grothendieck
		group $\go (\mathcal{M})$. Consequently, the categories of type $\mathcal M$ constructed
		above can not be used to realize these corresponding elliptic Lie
		algebras.
	\end{remark}

	\bibliographystyle{alpha}
	\bibliography{ref}

\end{document}